\documentclass[a4paper, 11pt]{scrartcl}
\usepackage{t1enc}
\usepackage{amsmath,amssymb,color,geometry}
\usepackage{wasysym}

\newtheorem{Lemma}{Lemma}

\newtheorem{Remark}{Remark}

\newenvironment{proof}[1][Proof]{\textbf{#1.} }{\hfill $\square$}

\newcommand{\cad}{c\`adl\`ag }
\newcommand{\what}{\widehat}
\newcommand{\wtil}{\widetilde}

\def \eps{\varepsilon}
\def \N{\mathbb{N}}
\def \R{\mathbb{R}}

\def \bM{\mathbb{M}}
\def \E{\mathbb{E}}
\def \F{\mathcal{F}}
\def \bF{\mathbb{F}}
\def \bL{\mathbb{L}}

\def \bH{\mathbb{H}}
\def \bD{\mathbb{D}}
\def \cE{\mathcal{E}}

\def \cU{\mathcal{U}}
\def \cT{\mathcal{T}}

\def \cM{\mathcal{M}}

\def \tOm{\widetilde{\Omega}}

\def \tpi{\widetilde{\pi}}

\def \al{\alpha}
\def \P{\mathbb{P}}
\def \1{\mathbf{1}}

\begin{document}

\title{$L^p$-solution for BSDEs with jumps in the case $p<2$ 
}
\subtitle{Corrections to the paper ``BSDEs with monotone generator driven by Brownian and Poisson noises in a general filtration''}

\author{T. Kruse \thanks{University of Duisburg-Essen, Thea-Leymann-Str. 9, 45127 Essen, Germany,
e-mail: \texttt{ thomas.kruse@uni-due.de}
}
, A. Popier \thanks{Universit\'e du
Maine, Laboratoire Manceau de Math\'ematiques, Avenue Olivier Messiaen, 72085 Le Mans, Cedex 9, France,
e-mail: \texttt{ alexandre.popier@univ-lemans.fr}
}
}
\date{\today}

\maketitle

\begin{abstract}
In \cite{krus:popi:14} we established existence and uniqueness of solutions of backward stochastic differential equations in $L^p$ under a monotonicity condition on the generator and in a general filtration. There was a mistake in the case $1 < p < 2$. Here we give a corrected proof. Moreover the quasi-left continuity condition on the filtration is removed.
\end{abstract}

\section*{Introduction}
%----------------

The aim of \cite{krus:popi:14} was to establish existence and uniqueness of solutions to BSDE in a general filtration that supports a Brownian motion $W$ and an independent Poisson random measure $\pi$. We considered the following multi-dimensional BSDE:
\begin{equation} \label{eq:gene_BSDE}
Y_t = \xi + \int_t^T f(s,Y_s, Z_s,\psi_s) ds - \int_t^T\int_\cU \psi_s(u) \tpi(du,ds) -\int_t^TZ_sdW_s- \int_t^T dM_s.
\end{equation}
Let us recall briefly the setting. We consider a filtered probability space $(\Omega,\F,\P,\bF = (\F_t)_{t\geq 0})$. The filtration is assumed to be complete and right continuous. We also assumed quasi-left continuity of the filtration. Nevertheless as mentioned in the introduction of \cite{bouc:poss:zhou:15} (see also Section 2.2 in \cite{popi:16}), this condition is unnecessary. 

The generator $f$ satisfies Conditions $\mathbf{ (H_{ex})}$\footnote{The precise definition of $\mathbf{ (H_{ex})}$ is given at the end of Section \ref{sect:function_space}.} in \cite{krus:popi:14}, that is, $f$ is Lipschitz continuous w.r.t.\ $z$ and $\psi$ and monotone w.r.t.\ $y$. On $\xi$ and $f^0_t = f(t,0,0,0)$, we keep the integrability condition: for some $p > 1$
\begin{equation} \label{eq:int_cond}
 \E \left(|\xi|^p +  \int_0^T |f(t,0,0,0)|^p dt \right) < +\infty.
 \end{equation}
Then the main results in \cite{krus:popi:14} can be summarized as follows. Under Assumptions $\mathbf{ (H_{ex})}$ and \eqref{eq:int_cond}, there exists a unique solution $(Y,Z,\psi,M)$ in $\cE^p(0,T)$ to the BSDE \eqref{eq:gene_BSDE} meaning that 
$$\E \left[  \sup_{t\in [0,T]} |Y_t|^p + \left( \int_0^T |Z_t|^2 dt \right)^{p/2}+  \left(  \int_0^T \int_{\cU} |\psi_s(u)|^2 \mu(du) ds \right)^{p/2}  + \left( [ M ]_T \right)^{p/2}  \right] < +\infty.$$
The comparison principle holds for this BSDE. If $ p\geq 2$, our results are true. But for $1 < p < 2$, as written in \cite{krus:popi:14} the main difference is that for $p< 2$ the compensator of a martingale does not control the predictable projection (see \cite{leng:lepi:prat:80} for a counterexample). In the proof of Proposition 3 in \cite{krus:popi:14}, Equality (31) does not hold in general. A simple counterexample is $Y_t = N_t - (T-t)$, $Z_t=0$, $\psi_t(u)=\1_{u=1}$, $\mu(du)=\delta_1(du)$.  Then 
$$Y_t = N_t - (T-t) = N_T - 2 \int_t^T ds - \int_t^T \int_\cU \psi_t(u) \tpi(du,ds).$$
Here the generator is $f(t,y,z,\psi) = -2\psi(1)$. In this case
\begin{eqnarray*}
&& \E  \int_{0}^{T}  \int_\cU |\psi_s(u)|^2  \left( |Y_{s-}|^2 \vee  |Y_{s-} +\psi_s(u)|^2 \right)^{p/2-1} \1_{|Y_{s-}|\vee |Y_{s-} + \psi_s(u)| \neq 0} \mu(du)ds \\
&& \quad = \E  \int_{0}^{T}  \left( |Y_{s-}|^2 \vee  |Y_{s-} +1|^2 \right)^{p/2-1} \1_{|Y_{s-}|\vee |Y_{s-} + 1| \neq 0} ds 
\end{eqnarray*}
where $Y_{s-}$ is the left limit of $Y$ at time $s$, and
\begin{eqnarray*}
&& \E  \int_{0}^{T}  \int_\cU |\psi_s(u)|^2 |Y_{s}|^{p-2} \1_{|Y_{s}| \neq 0} \mu(du)ds =  \E  \int_{0}^{T}  |Y_{s}|^{p-2}  \1_{|Y_{s}| \neq 0} ds.
\end{eqnarray*}
For $1<p<2$, the second integral is strictly greater than the first one. In other words, if the generator does not depend on $\psi$, our earlier proof in \cite{krus:popi:14} is safe (see also \cite{klim:15} and \cite{klim:rozk:13}). But the dependance due to the generator cannot be controlled by the first integral if $p< 2$. Thereby Proposition 3, Theorem 2, Propositions 5 and 6 and Theorem 3 in \cite{krus:popi:14} are not proved when $p < 2$. Here we present proofs for these results under strengthened assumptions.

A $\psi$-depending non trivial generator $f$ can be found in \cite{krus:popi:15} (see BSDE (3) in this paper). This example is coming from an optimal stochastic control problem. It follows from the proof of Corollary 1 in \cite{krus:popi:15} that Condition $(\mathbf{ H_{comp}})$ is satisfied (see Section \ref{ssect:comp} below) and thus $(\mathbf{ H_{ex}})$ and \textbf{(C)} are satisfied as well (see Lemma \ref{lem:hcomp->hex} and the proof of Theorem 2). Many other examples can be found for example in \cite{delo:13}, Part II (see among others BSDEs (9.30) or (11.13)).

\section{Choice of a suitable function space for the Poisson integrand when $p<2$} \label{sect:function_space}
%---------------

Recall briefly the notations of \cite{krus:popi:14}. We consider a filtered probability space $(\Omega,\F,\P,\bF = (\F_t)_{t\geq 0})$, the filtration being complete and right continuous. 
Without loss of generality we suppose that all semimartingales have right continuous paths with left limits and we assume that $(\Omega,\F,\P,\bF = (\F_t)_{t\geq 0})$ supports a $k$-dimensional Brownian motion $W$ and a Poisson random measure $\pi$ with intensity $\mu(du)dt$ on the space $\cU \subset \R^m \setminus \{0\} $. The measure $\mu$ is $\sigma$-finite on $\cU$ such that
$$\int_\cU (1\wedge |u|^2) \mu(du) <+\infty.$$
The compensated Poisson random measure $\tpi(du,dt) = \pi(du,dt) - \mu(du) dt$ is a martingale w.r.t.\ the filtration $\bF$. Moreover we introduce the following notations.

\begin{itemize}
\item $G_{loc}(\mu)$ is the set of predictable functions $\psi$ on $\tOm= \Omega \times [0,T] \times \cU$ such that for any $t \geq 0$ a.s.
$$ \int_0^t \int_\cU (|\psi_s(u)|^2\wedge |\psi_s(u)|) \mu(du) < +\infty.$$
\item $ \cM_{loc}$ is the set of c\`adl\`ag local martingales orthogonal to $W$ and $\tpi$. $\cM$ is the subspace of $ \cM_{loc}$ of martingales.
\item $\bD^p(0,T)$ is the space of all adapted \cad processes $X$ such that
$\E \left(  \sup_{t\in [0,T]} |X_t|^p \right) < +\infty.$
For simplicity, we write $X_* = \sup_{t\in [0,T]} |X_t|$ and $X_*^{\beta,p}=\sup_{t\in [0,T]} e^{\beta t} |X_t|^p$. 
\item $\bH^p(0,T)$ is the subspace of all predictable processes $X$ such that
$\E \left[ \left( \int_0^T |X_t|^2 dt\right)^{p/2} \right] < +\infty.$
\item $\bM^p(0,T)$ is the subspace of $\cM$ of all martingales such that
$\E \left[ \left( [ M ]_T \right)^{p/2}\right] < +\infty.$
\item $\bL^p_\pi(0,T) = \bL^p_{\pi}(\Omega\times (0,T)\times \cU)$ is the set of processes $\psi \in G_{loc}(\mu)$ such that
$$\E \left[ \left(  \int_0^T \int_{\cU} |\psi_s(u)|^2 \pi(du,ds) \right)^{p/2} \right] < +\infty .$$
\item $\bL^p_\mu=\bL^p(\cU,\mu;\R^d)$ is the set of measurable functions $\psi : \cU \to \R^d$ such that
$\| \psi \|^p_{\bL^p_\mu} = \int_{\cU} |\psi(u)|^p \mu(du)  < +\infty .$
\item $\cE^p(0,T) = \bD^p(0,T) \times \bH^p(0,T) \times \bL^p_\pi(0,T) \times \bM^p(0,T)$.
\end{itemize}

Let $\psi \in G_{loc}(\mu)$. Let us recall known results on the (local) martingale $N$ given by 
\begin{equation} \label{eq:poiss_mart}
N_t =   \int_0^t  \int_{\cU} \psi_s(u) \tpi(du, ds), t\ge 0.
\end{equation}
It follows that the compensator is given by 
$$ [N]_t =  \int_0^t  \int_{\cU} |\psi_s(u)|^2 \pi(du, ds).$$
For $t\ge 0$ let
$$N_t^* = \sup_{r\in [0,t]} \left|  \int_0^r  \int_{\cU} \psi_s(u) \tpi(du, ds) \right|.$$

Now let us define the following norm on $\psi \in  G_{loc}(\mu)$: if $\nu$ is the measure defined on $\cU \times [0,T]$ by $\nu = \mu \otimes \mbox{Leb}$, then
\begin{eqnarray*}
\|\psi\|_{\bL^p(\bL^2_{\nu}) + \bL^p(\bL^p_{\nu})}  =  \inf_{\psi^1+\psi^2 = \psi} \left\{ \left(  \E \left[ \left( \int_0^T  \int_{\cU} |\psi^1_s(u)|^2 \mu(du)ds \right)^{p/2} \right]\right)^{1/p} \right.
 + \left. \left(\E \left[ \int_0^T  \int_{\cU} |\psi^2_s(u)|^p \mu(du)ds \right]\right)^{1/p} \right\}.
\end{eqnarray*}
And for $p\in [1,2)$ and for a measurable function $\phi$ defined on $\cU$, we put
$$\|\phi \|_{\bL^p_\mu+\bL^2_\mu}  =  \inf_{\phi^1+\phi^2 = \phi} \left( \|\phi^1\|_{\bL^p_\mu} +  \|\phi^2\|_{\bL^2_\mu} \right).$$
With this norm we can define the Banach space $\bL^p_\mu+ \bL^2_\mu$ (for the definition of the sum of two Banach spaces, see for example \cite{krei:petu:seme:82}). By the same way we define $\bL^p_\nu+ \bL^2_\nu$.  
\begin{Lemma} \label{lem:class_ineq}
$\ $
\begin{itemize}
\item \textbf{Burkholder-Davis-Gundy inequality:} For all $p\in [1,\infty)$ there exist two universal constants $c_p$ and $C_p$ (not depending on $N$) such that for any $N$ defined by \eqref{eq:poiss_mart} and for any $t \geq 0$
\begin{equation} \label{eq:BDG_ineq}
c_p \E \left( [N]_t^{p/2} \right) \leq \E \left[ (N^*_t)^{p} \right] \leq C_p \E \left( [N]_t^{p/2} \right).
\end{equation}
\item \textbf{Bichteler-Jacod inequality\footnote{See the discussion in \cite{mari:rock:14} for the name of this estimate.}:} For $p \in (1,2)$, there exist two universal constants $\kappa_p$ and $K_p$ such that for any $\psi \in G_{loc}(\mu)$, if $N$ is defined by \eqref{eq:poiss_mart}
\begin{equation} \label{eq:BJ_ineq}
\kappa_p\left[ \E \left( [N]_T^{p/2} \right) \right]^{1/p} \leq \|\psi\|_{\bL^p(\bL^2_\nu) + \bL^p(\bL^p_\nu)} \leq K_p  \left[ \E \left( [N]_T^{p/2} \right) \right]^{1/p}.
\end{equation}
\end{itemize}
\end{Lemma}
\begin{proof}
The first inequality \eqref{eq:BDG_ineq} is proved in \cite{jaco:79}, Proposition 3.66. The second result \eqref{eq:BJ_ineq} can be found for example in \cite{mari:rock:14}, Theorem 1 and the following comments pages 297 and 298. 
\end{proof}

From the Bichteler-Jacod inequality \eqref{eq:BJ_ineq} we deduce the next result.
\begin{Lemma} \label{lem:sum_norm_estim}
For $p \in (1,2)$, there exists a universal constant $K_{p,T}$ such that for any $\psi \in G_{loc}(\mu)$ and $N$ defined by \eqref{eq:poiss_mart}
\begin{equation} \label{eq:norm_equiv}
\E \left[  \int_0^T  \|\psi_s\|^p_{\bL^p_\mu+\bL^2_\mu} ds\right] \leq K_{p,T}   \E \left( [N]_T^{p/2} \right) .
\end{equation}
If a function $\phi$ defined on $[0,T]\times \cU$ is in $\bL^1_\nu+ \bL^2_\nu$, then
\begin{equation} \label{eq:int_norm_sum_banach}
\int_0^T \|\phi_s\|_{\bL^1_\mu + \bL^2_\mu} ds \leq (1\vee \sqrt{T}) \|\phi\|_{\bL^1_\nu + \bL^2_\nu}.
\end{equation}
\end{Lemma}
\begin{proof}
Let $\psi^1 \in \bL^p(\bL^2_\nu)$ and $\psi^2 \in \bL^p(\bL^p_\nu)$ such that $\psi =\psi^1+\psi^2$. By Jensen's inequality:
\begin{eqnarray*}
\E \left[  \int_0^T \|\psi^1_s\|^p_{\bL^2_\mu} ds\right] & = & \E \left[  \int_0^T  \left( \int_{\cU} |\psi^1_s(u)|^2 \mu(du)\right)^{p/2} ds  \right] \\
 & \leq & T^{1-\frac{p}{2}}  \E \left[ \left( \int_0^T  \int_{\cU} |\psi^1_s(u)|^2 \mu(du)ds \right)^{p/2} \right] =T^{1-\frac{p}{2}}  \|\psi^1\|^p_{\bL^p(\bL^2_\nu)}. 
\end{eqnarray*}
and
\begin{eqnarray*}
 \E \left[  \int_0^T \|\psi^2_s\|^p_{\bL^p_\mu} ds\right] & = & \E \left[  \int_0^T  \left( \int_{\cU} |\psi^2_s(u)|^p \mu(du)\right) ds  \right] = \|\psi^2\|^p_{\bL^p(\bL^p_\nu)}
\end{eqnarray*}
We deduce \eqref{eq:norm_equiv} directly from Bichteler-Jacod inequality \eqref{eq:BJ_ineq}.

For the second inequality, if $\phi \in \bL^1_\nu+ \bL^2_\nu$, for any $\eps > 0$, there are two functions $\phi^1$ and $\phi^2$ in $\bL^1_\nu$, respectively  in $\bL^2_\nu$ with $\phi = \phi^1+\phi^2$ and 
$$\|\phi\|_{\bL^1_\nu + \bL^2_\nu}\leq \|\phi^1\|_{\bL^1_\nu} +\|\phi^2\|_{\bL^2_\nu} \leq \|\phi\|_{\bL^1_\nu + \bL^2_\nu} + \eps.$$ 
Hence for almost any $s \in [0,T]$, $\phi^1_s$ (resp. $\phi^2_s$)  belongs to $\bL^1_\mu$ (resp. $\bL^2_\mu$) with $\phi_s = \phi^1_s+\phi^2_s$. Thus by definition $\|\phi_s\|_{\bL^1_\mu + \bL^2_\mu} \leq \|\phi^1_s\|_{\bL^1_\mu} + \|\phi^2_s\|_{\bL^2_\mu}$. We integrate this inequality between 0 and $T$ and by Jensen's inequality
\begin{eqnarray*}
\int_0^T \|\phi_s\|_{\bL^1_\mu + \bL^2_\mu} ds & \leq & \int_0^T\|\phi^1_s\|_{\bL^1_\mu} ds + \int_0^T \|\phi^2_s\|_{\bL^2_\mu} ds \\
& \leq  & \|\phi^1\|_{\bL^1_\nu} +\sqrt{T} \|\phi^2\|_{\bL^2_\nu} \leq  (1\vee \sqrt{T})(\|\phi^1\|_{\bL^1_\nu} + \|\phi^2\|_{\bL^2_\nu}) \\
& \leq & (1\vee \sqrt{T})(\|\phi\|_{\bL^1_\nu + \bL^2_\nu} + \eps).
\end{eqnarray*}
Since these inequalities are true for any $\eps > 0$, we deduce Estimate \eqref{eq:int_norm_sum_banach}.
\end{proof}

In particular \eqref{eq:BDG_ineq} means that the martingale $N$ is well-defined (see Chapter II in \cite{jaco:79}) provided we can control $[N]$ in $\bL^{p/2}(\Omega)$. And from \eqref{eq:norm_equiv}, $\P \otimes \mbox{Leb}$-a.s. on $\Omega \times [0,T]$, $\psi_t$ is in $\bL^p_\mu+\bL^2_\mu$ if again we control $[N]$ in $\bL^{p/2}(\Omega)$. From Lemma \ref{lem:sum_Lp_spaces} below, $\psi_t$ is also in $\bL^1_\mu+\bL^2_\mu$ and this implies that for any $b \in (0,+\infty)$ 
$$\E \left[ \int_0^T  \int_{\cU} \left( |\psi_s(u)|^2 \1_{|\psi_s(u)|\leq b} + |\psi_s(u)| \1_{|\psi_s(u)|> b}\right) \mu(du) ds\right] < +\infty.$$
This last estimate can be also found in Proposition 3.68 of \cite{jaco:79}\footnote{With our setting, the process $\widehat W$ of \cite{jaco:79} is identically equal to zero.}  in a more general setting.

To illustrate and motivate our purpose, let us consider a stable L\'evy process $X=(X_t, \ 0\leq t \leq T)$. The L\'evy measure is $\mu(du) = \frac{1}{|u|^{1+\alpha}} du$ where $u \in \cU= \R\setminus \{0\}$ and $0 < \alpha < 2$. Then by the L\'evy-Khintchine decomposition:
\begin{eqnarray*}
X_t & = & \int_0^t \int_\cU u \1_{|u|<1} \tpi(du,ds) +  \int_0^t \int_\cU u \1_{|u|\geq1} \pi(du,ds) \\
& = &  \int_0^t \int_\cU u \tpi(du,ds) + t \int_\cU u \1_{|u|\geq1} \mu(du) =  \int_0^t \int_\cU u \tpi(du,ds) .
\end{eqnarray*}
Now $X_T \in \bL^p(\Omega)$ if and only if $p<\alpha <2$. We take $\xi = X_T$, $Y_t = X_t$, $\psi_t(u) = u$ and
$$Y_t = \xi - \int_t^T \int_\cU u \tpi(du,ds).$$ 
For any $t \in [0,T]$, $p<\alpha$, $Y_t$ is in $\bL^p(\Omega)$ and $\psi_t \not\in \bL^2_\mu$. Nevertheless for any $\delta > 0$, $\phi^1_t = \psi_t \1_{|\psi_t|\leq \delta}$ belongs to $\bL^2_\mu$ and $\phi^2_t=\psi_t\1_{|\psi_t|\geq \delta}$ to $\bL^p_\mu$. Thus $\psi_t$ is in $\bL^p_\mu+\bL^2_\mu$. And it is easy to check that $\psi_t$ also belongs to $\bL^1_\mu+\bL^2_\mu$.

\vspace{0.5cm}
\noindent \textbf{Conclusion and assumption on $f$:} From now on, we assume that $p \in (1,2)$. Then we have to choose $\psi$ in a suitable integrability space, namely $\bL^1_\mu+ \bL^2_\mu$. From the next Lemma \ref{lem:sum_Lp_spaces}, this space contains all spaces $\bL^p_\mu + \bL^2_\mu$. Hence in the rest of the paper, our generator $f$ satisfies \textbf{Condition} $\mathbf{ (H_{ex})}$: 
\begin{enumerate}
\item[(H1)] For every $t \in [0,T]$, $z\in \R^{d\times k}$ and every $\psi \in \bL^1_\mu+\bL^2_\mu$ the mapping $y \in \R^d \mapsto f(t,y,z,\psi)$ is continuous. Moreover there exists a constant
$\alpha$ such that
$$\langle f(t,y,z,\psi)-f(t,y',z,\psi),y-y'\rangle \leq \alpha |y-y'|^2.$$
\item[(H2)] For every $r  > 0$ the mapping $(\omega,t) \mapsto \sup_{|y|\leq r} |f(t,y,0,0)-f(t,0,0,0)|$ belongs to $L^1(\Omega \times [0,T], \P \otimes m)$.
\item[(H3)] There exists a constant $K$ such that for any $t$ and $y$, for any $z$, $z'$ in $\R^{d\times k}$ and $\psi$, $\psi'$ in $\bL^1_\mu+ \bL^2_\mu$
\begin{equation*} %\tag{H3} \label{eq:H3}
|f(t,y,z,\psi) - f(t,y,z',\psi')|\leq K \left( |z-z'| + \|\psi-\psi'\|_{\bL^1_\mu+ \bL^2_\mu} \right).
\end{equation*}
\end{enumerate}
Note that (H1) and (H2) coincide with assumptions (H1) and (H2) in \cite{krus:popi:14}, whereas (H3) above replaces the older condition (H3) in  \cite{krus:popi:14}.

\begin{Remark}
Note that if $p\geq 2$, since $\bL^p_\mu+ \bL^2_\mu \subset \bL^2_\mu$, the generator $f$ can be defined on the function set $\bL^2_\mu$. 
\end{Remark}

The next result will be used several times. Although it is quite simple we did not find a reference. The proof is postponed to the end of this paper.
\begin{Lemma} \label{lem:sum_Lp_spaces}
Let $p\in [1,\infty)$ and  $\phi : \cU \to \R^d$ be a function in $\bL^p_\mu + \bL^2_\mu$. Then
$\|\phi \|_{\bL^p_\mu+\bL^2_\mu} < +\infty$ if and only if for any $\delta > 0$ it holds that $\phi \1_{|\phi| \leq \delta} \in \bL^2_\mu$ and $\phi \1_{|\phi| > \delta} \in \bL^p_\mu$. Moreover it holds that $\bL^p_\mu+\bL^2_\mu \subset \bL^1_\mu+\bL^2_\mu$. The same results hold if $\mu$ is replaced by $\nu$.
\end{Lemma}

\section{Complete proof for $p<2$ of Proposition 3 and Theorem 2}
%------------

We say that the Condition \textbf{(C)} holds if $\P$-a.s.
$$\langle  \check{y}, f(t,y,z,\psi) \rangle \leq f_t + \al |y| + K |z|+ K \|\psi\|_{\bL^1_\mu + \bL^2_\mu},$$
with $K \geq 0$ and $f_t$ is a non-negative progressively measurable process. Note that compared to \cite{krus:popi:14}, we change the norm on $\psi$. Recall that for $y\in \R^d\setminus\{0\}$ we write $\check{y}=\frac{1}{|y|}y$ and $\check 0 =0$. Let us denote $F = \int_0^T f_r dr$.

\vspace{0.5cm}
\noindent \textbf{ Proposition 3} 
\textit{ Let the Condition} \textbf{(C)} \textit{hold and let be $(Y,Z,\psi,M) \in \cE^p(0,T)$ be a solution of BSDE \eqref{eq:gene_BSDE} and assume moreover that $F^p$ is integrable. Then there exists a constant $C$ depending on $p$, $K$ and $T$ such that }
\begin{eqnarray*}
\E \left[\sup_{t\in[0,T]}   |Y_t|^p +  \left( \int_0^T  |Z_t|^2 dt \right)^{p/2} + \left(  [ M ]_T \right)^{p/2}+  \left(  \int_0^T  \int_{\cU} |\psi_s(u)|^2 \pi(du, ds) \right)^{p/2}\right] 
\leq C \E \left[|\xi|^p  + \left( \int_0^T f_r dr \right)^p \right].
\end{eqnarray*}
\vspace{0.3cm}

A comment before the proof. If $\psi \in \bL^p_\pi$, then Inequality \eqref{eq:BJ_ineq} and Lemma \ref{lem:sum_norm_estim} imply that $\P$-a.s. $\psi \in \bL^p_\nu + \bL^2_\nu$ and from Lemma \ref{lem:sum_Lp_spaces}, $\psi \in \bL^1_\nu + \bL^2_\nu$. Hence the integrand $\psi$ is in the required function space (see Condition (H3)).

\vspace{0.5cm}
\noindent \begin{proof}

\noindent \textbf{ Step 1:} We prove first that there exist two constants $\beta$ (depending on $K$, $\al$ and $p$) and $\kappa_{p,\beta}$ such that 
\begin{eqnarray} \nonumber
&&\E \int_{0}^{T} e^{\beta s} |Y_{s}|^{p-2} |Z_s|^2 \1_{Y_s\neq 0} ds + \E \int_{0}^{T} e^{\beta s} |Y_{s-}|^{p-2}  \1_{Y_{s-}\neq 0} d[ M ]^c_s \\ \nonumber
&& +  \E \int_{0}^{T}  \int_\cU e^{\beta s} |\psi_s(u)|^2  \left( |Y_{s-}|^2 \vee  |Y_{s-} +\psi_s(u)|^2 \right)^{p/2-1} \1_{|Y_{s-}|\vee |Y_{s-} + \psi_s(u)| \neq 0}\pi(du,ds) \\  \nonumber
&& +  \E \sum_{0< s \leq T} e^{\beta s}  \left( |Y_{s-}|^2 \vee  |Y_{s-} + \Delta M_s|^2 \right)^{p/2-1}  \1_{|Y_{s-}|\vee |Y_{s-} + \Delta M_s| \neq 0} |\Delta M_s|^2\\   \label{eq:Lp_estim_p_leq_2_mart}
&& + \E \int_{0}^{T} e^{\beta s}  |Y_s|^p  ds \leq \kappa_{p,\beta}  \E(X).
\end{eqnarray}
where 
$$ X = e^{\beta T} |\xi|^p +   p \int_0^T e^{\beta s} |Y_s|^{p-1} f_s ds.$$
In the following let $c(p) = p(p-1)/2$. For some constant $\beta\in \R$, we apply It\^o's formula (see Corollary 1 in \cite{krus:popi:14}) for $\tau \in \cT_T$ to $e^{\beta t} |Y_{t}|^p$ to obtain
\begin{eqnarray*}
&& e^{\beta (t\wedge \tau)} |Y_{t\wedge \tau}|^p + c(p) \int_{t\wedge \tau}^{\tau} e^{\beta s}  |Y_{s}|^{p-2} |Z_s|^2 \1_{Y_s\neq 0} ds + c(p) \int_{t\wedge \tau}^{\tau} e^{\beta s}  |Y_{s-}|^{p-2}  \1_{Y_{s-}\neq 0} d[ M ]^c_s \\
&& \leq e^{\beta \tau}|Y_{\tau}|^p+ p \int_{t\wedge \tau}^{\tau}e^{\beta s} |Y_{s-}|^{p-1} \check{Y}_{s-} f(s,Y_s,Z_s) ds -\beta  \int_{t\wedge \tau}^{\tau}e^{\beta s} |Y_s|^{p} ds \\
&&\quad -  p \int_{t\wedge \tau}^{\tau} e^{\beta s} |Y_{s-}|^{p-1} \check{Y}_{s-} Z_s dW_s    -  p  \int_{t\wedge \tau}^{\tau} e^{\beta s} |Y_{s-}|^{p-1} \check{Y}_{s-} dM_s\\
&&\quad  -  p \int_{t\wedge \tau}^{\tau} e^{\beta s} |Y_{s-}|^{p-1} \check{Y}_{s-}  \int_\cU \psi_s(u) \tpi(du,ds)  \\
& &\quad -  \int_{t\wedge \tau}^{\tau} e^{\beta s}  \int_\cU \left[ |Y_{s-}+\psi_s(u)|^p -|Y_{s-}|^p - p|Y_{s-}|^{p-1} \check{Y}_{s-} \psi_s(u) \right] \pi(du,ds) \\
&& \quad - \sum_{t\wedge \tau< s \leq \tau} e^{\beta s}  \left[ |Y_{s-}+\Delta M_s|^p - |Y_{s-}|^p - p|Y_{s-}|^{p-1}  \check{Y}_{s-} \Delta M_s \right] .
\end{eqnarray*}
With Condition (C) this becomes
\begin{eqnarray*}
&&e^{\beta (t\wedge \tau)} |Y_{t\wedge \tau}|^p + c(p) \int_{t\wedge \tau}^{\tau} e^{\beta s} |Y_{s}|^{p-2} |Z_s|^2 \1_{Y_s\neq 0} ds + c(p) \int_{t\wedge \tau}^{\tau} e^{\beta s} |Y_{s-}|^{p-2}  \1_{Y_{s-}\neq 0} d[ M ]^c_s \\
&& \leq e^{\beta \tau}  |Y_{\tau}|^p + \int_{t\wedge \tau}^{\tau} e^{\beta s}\left( p|Y_s|^{p-1} f_s + (p \al-\beta) |Y_s|^p  \right) ds + pK \int_{t\wedge \tau}^{\tau} e^{\beta s}|Y_s|^{p-1} |Z_s| ds  \\
&& \qquad + pK \int_{t\wedge \tau}^{\tau} e^{\beta s}|Y_{s-}|^{p-1} \|\psi_s\|_{\bL^1_\mu + \bL^2_\mu} ds  
-  p \int_{t\wedge \tau}^{\tau} e^{\beta s}|Y_s|^{p-1} \check{Y}_s Z_s dW_s  \\
& &\qquad  -  p  \int_{t\wedge \tau}^{\tau} e^{\beta s}|Y_{s-}|^{p-1} \check{Y}_{s-} dM_s -  p \int_{t\wedge \tau}^{\tau} e^{\beta s}|Y_{s-}|^{p-1} \check{Y}_{s-}  \int_\cU \psi_s(u) \tpi(du,ds)  \\
& &\quad -  \int_{t\wedge \tau}^{\tau} e^{\beta s} \int_\cU \left[ |Y_{s-}+\psi_s(u)|^p -|Y_{s-}|^p - p|Y_{s-}|^{p-1}  \check{Y}_{s-} \psi_s(u) \right] \pi(du,ds) \\
&& \quad - \sum_{t\wedge \tau< s \leq \tau} e^{\beta s} \left[ |Y_{s-}+\Delta M_s|^p - |Y_{s-}|^p - p|Y_{s-}|^{p-1} \check{Y}_{s-}  \Delta M_s \right].
\end{eqnarray*}
Moreover by Young's inequality
 \begin{equation*}
pKe^{\beta s}  |Y_s|^{p-1} |Z_s| \leq \frac{pK^2}{p-1}e^{\beta s} |Y_s|^p + \frac{c(p)}{2} e^{\beta s} |Y_{s}|^{p-2} |Z_s|^2 \1_{Y_s\neq 0}. 
\end{equation*}
Hence we obtain that
\begin{eqnarray} \nonumber
&&e^{\beta (t\wedge \tau)} |Y_{t\wedge \tau}|^p + \frac{c(p)}{2} \int_{t\wedge \tau}^{\tau} e^{\beta s} |Y_{s}|^{p-2} |Z_s|^2 \1_{Y_s\neq 0} ds + c(p) \int_{t\wedge \tau}^{\tau} e^{\beta s} |Y_{s-}|^{p-2}  \1_{Y_{s-}\neq 0} d[ M ]^c_s \\ \nonumber
&& \leq e^{\beta \tau}  |Y_{\tau}|^p + \int_{t\wedge \tau}^{\tau} e^{\beta s}\left( p|Y_s|^{p-1} f_s + (p \al + \frac{pK^2}{p-1}-\beta) |Y_s|^p  \right) ds   \\ \nonumber
&& \qquad + pK \int_{t\wedge \tau}^{\tau} e^{\beta s}|Y_{s-}|^{p-1} \|\psi_s\|_{\bL^1_\mu + \bL^2_\mu} ds  
-  p \int_{t\wedge \tau}^{\tau} e^{\beta s}|Y_s|^{p-1} \check{Y}_s Z_s dW_s  \\ \nonumber
& &\qquad  -  p  \int_{t\wedge \tau}^{\tau} e^{\beta s}|Y_{s-}|^{p-1} \check{Y}_{s-} dM_s -  p \int_{t\wedge \tau}^{\tau} e^{\beta s}|Y_{s-}|^{p-1} \check{Y}_{s-}  \int_\cU \psi_s(u) \tpi(du,ds)  \\ \nonumber
& &\quad -  \int_{t\wedge \tau}^{\tau} e^{\beta s} \int_\cU \left[ |Y_{s-}+\psi_s(u)|^p -|Y_{s-}|^p - p|Y_{s-}|^{p-1}  \check{Y}_{s-} \psi_s(u) \right] \pi(du,ds) \\ \label{eq:first_tech_ineq}
&& \quad - \sum_{t\wedge \tau< s \leq \tau} e^{\beta s} \left[ |Y_{s-}+\Delta M_s|^p - |Y_{s-}|^p - p|Y_{s-}|^{p-1} \check{Y}_{s-}  \Delta M_s \right].
\end{eqnarray}
In particular we have:
\begin{eqnarray} \nonumber
&& 0 \leq \int_{t\wedge \tau}^{\tau} e^{\beta s} \int_\cU \left[ |Y_{s-}+\psi_s(u)|^p -|Y_{s-}|^p - p|Y_{s-}|^{p-1}  \check{Y}_{s-} \psi_s(u) \right] \pi(du,ds) \\ \nonumber
&& \quad \leq e^{\beta \tau}  |Y_{\tau}|^p + \int_{t\wedge \tau}^{\tau} e^{\beta s}\left( p|Y_s|^{p-1} f_s + (p \al + \frac{pK^2}{p-1}-\beta) |Y_s|^p  \right) ds   \\ \nonumber
&& \qquad + pK \int_{t\wedge \tau}^{\tau} e^{\beta s}|Y_{s-}|^{p-1} \|\psi_s\|_{\bL^1_\mu + \bL^2_\mu} ds  
-  p \int_{t\wedge \tau}^{\tau} e^{\beta s}|Y_s|^{p-1} \check{Y}_s Z_s dW_s  \\ \label{eq:control_integral_pi}
& &\qquad  -  p  \int_{t\wedge \tau}^{\tau} e^{\beta s}|Y_{s-}|^{p-1} \check{Y}_{s-} dM_s -  p \int_{t\wedge \tau}^{\tau} e^{\beta s}|Y_{s-}|^{p-1} \check{Y}_{s-}  \int_\cU \psi_s(u) \tpi(du,ds),
\end{eqnarray}
where the first inequality is due to convexity. Now since $(Y,\psi)\in \bD^p \times \bL^p_\pi$, Inequalities \eqref{eq:BJ_ineq} and \eqref{eq:int_norm_sum_banach} and Young's inequality give:
\begin{eqnarray*} 
pK \E \int_{0}^{T} e^{\beta s}|Y_{s-}|^{p-1} \|\psi_s\|_{\bL^1_\mu + \bL^2_\mu} ds 
&\leq& pK \E \left[ \sup_{s\in [0,T]} \left( e^{\beta (p-1) s/p}|Y_{s-}|^{p-1} \right) \int_{0}^{T} e^{\beta s/p} \|\psi_s\|_{\bL^1_\mu + \bL^2_\mu} ds \right] \\
& \leq & K(p-1 ) \E \left[ \sup_{s\in [0,T]} \left( e^{\beta s}|Y_{s}|^{p} \right) \right] + K \E \left[ \left( \int_{0}^{T} e^{\beta s/p} \|\psi_s\|_{\bL^1_\mu + \bL^2_\mu} ds\right)^p \right]< +\infty
\end{eqnarray*}
and since $F^p$ is integrable
\begin{eqnarray*} 
&& \E  \int_{0}^{T} e^{\beta s} |Y_s|^{p-1} f_s ds \leq  \E \left[ \sup_{s\in [0,T]} \left( e^{\beta s}|Y_{s}|^{p} \right) \right] +  \E \left[ \left( \int_{0}^{T} e^{\beta s/p} f_s ds\right)^p \right]< +\infty.
\end{eqnarray*}
Using a fundamental sequence of stopping times $\tau_k$ for the local martingale 
$$\int_{0}^{.} e^{\beta s} |Y_{s-}|^{p-1} \check{Y}_{s-} \left( Z_s dW_s  + dM_s +  \int_\cU \psi_s(u) \tpi(du,ds) \right)$$
and taking $\tau=\tau_k$ and the expectation in \eqref{eq:control_integral_pi}, this local martingale term will disappear in \eqref{eq:control_integral_pi}. Then since $Y \in \bD^p(0,T)$, by monotone convergence theorem we obtain when $k$ goes to $\infty$
\begin{equation} \label{eq:expectation_pi}
0\leq \E  \int_{0}^{T} e^{\beta s} \int_\cU \left[ |Y_{s-}+\psi_s(u)|^p -|Y_{s-}|^p - p|Y_{s-}|^{p-1}  \check{Y}_{s-} \psi_s(u) \right] \pi(du,ds) < +\infty.
\end{equation}

From Lemma \ref{lem:tech_inequality} we choose $\eps > 0$ depending on $p$ and $K$ (see \eqref{eq:estim_eps} for a possible choice of $\eps$) and we fix $\delta = \vartheta(\eps,p)|Y_{s-}|$ if $Y_{s-}\neq 0$ (and any $\delta > 0$ if $Y_{s-}=0$) where $\vartheta$ is defined in Lemma \ref{lem:tech_inequality}. From the norm definition on $\bL^1_\mu + \bL^2_\mu$ and Young's inequality it follows that
\begin{eqnarray*}
pK \int_{t\wedge \tau}^{\tau} e^{\beta s}|Y_{s-}|^{p-1} \|\psi_s\|_{\bL^1_\mu + \bL^2_\mu} ds 
&\leq & pK\int_{t\wedge \tau}^{\tau} e^{\beta s} |Y_{s-}|^{p-1} \left(  \|\psi_s \1_{|\psi_s| < \delta} \|_{\bL^2_\mu} + \|\psi_s \1_{|\psi_s|\geq \delta} \|_{\bL^1_\mu}\right) ds  \\
 &\leq& \frac{p K^2}{2\eps}\int_{t\wedge \tau}^{\tau} e^{\beta s} |Y_{s-}|^{p} ds +\frac{p\eps}{2} \int_{t\wedge \tau}^{\tau} e^{\beta s} |Y_{s-}|^{p-2} \1_{Y_{s-} \neq 0} \|\psi_s  \1_{|\psi_s| < \delta}  \|^2_{\bL^2_\mu} ds \\
&& \qquad + p K\int_{t\wedge \tau}^{\tau} e^{\beta s} |Y_{s-}|^{p-1} \|\psi_s \1_{|\psi_s|\geq \delta} \|_{\bL^1_\mu}ds.
\end{eqnarray*}
From Lemma 9 in \cite{krus:popi:14}
 \begin{eqnarray*}
&& \int_{t\wedge \tau}^{\tau}  \int_\cU e^{\beta s} \left[ |Y_{s-}+\psi_s(u)|^p -|Y_{s-}|^p - p|Y_{s-}|^{p-1}  \check{Y}_{s-} \psi_s(u) \right] \pi(du,ds) \\
&& \quad \geq c(p) \int_{t\wedge \tau}^{\tau}  \int_\cU e^{\beta s} |\psi_s(u)|^2  \left( |Y_{s-}|^2 \vee  |Y_{s-} +\psi_s(u)|^2 \right)^{p/2-1} \1_{|Y_{s-}|\vee |Y_{s-} + \psi_s(u)| \neq 0} \pi(du,ds).
\end{eqnarray*}
and
 \begin{eqnarray*}
&&\sum_{t\wedge \tau< s \leq \tau} e^{\beta s} \left[ |Y_{s-}+\Delta M_s|^p - |Y_{s-}|^p - p|Y_{s-}|^{p-1} \check{Y}_{s-} \Delta M_s \right]  \\
&& \quad \geq c(p)  \sum_{t\wedge \tau< s \leq \tau} e^{\beta s}|\Delta M_s|^2  \left( |Y_{s-}|^2 \vee  |Y_{s-} + \Delta M_s|^2 \right)^{p/2-1}  \1_{|Y_{s-}|\vee |Y_{s-} + \Delta M_s| \neq 0}
\end{eqnarray*}
Therefore from Inequality \eqref{eq:first_tech_ineq} we deduce the following inequality for any $\eps \in (0,+\infty)$
\begin{eqnarray} \nonumber
&&e^{\beta (t\wedge \tau)} |Y_{t\wedge \tau}|^p + \frac{c(p)}{2} \int_{t\wedge \tau}^{\tau} e^{\beta s} |Y_{s}|^{p-2} |Z_s|^2 \1_{Y_s\neq 0} ds + c(p) \int_{t\wedge \tau}^{\tau} e^{\beta s} |Y_{s-}|^{p-2}  \1_{Y_{s-}\neq 0} d[ M ]^c_s \\ \nonumber
&& + \frac{c(p)}{2}  \int_{t\wedge \tau}^{\tau}  \int_\cU e^{\beta s} |\psi_s(u)|^2  \left( |Y_{s-}|^2 \vee  |Y_{s-} +\psi_s(u)|^2 \right)^{p/2-1} \1_{|Y_{s-}|\vee |Y_{s-} + \psi_s(u)| \neq 0}\pi(du,ds) \\  \nonumber
&& + c(p) \sum_{t\wedge \tau< s \leq \tau} e^{\beta s}  \left( |Y_{s-}|^2 \vee  |Y_{s-} + \Delta M_s|^2 \right)^{p/2-1}  \1_{|Y_{s-}|\vee |Y_{s-} + \Delta M_s| \neq 0} |\Delta M_s|^2\\  \nonumber
&& \leq e^{\beta \tau}  |Y_{\tau}|^p + p \int_{t\wedge \tau}^{\tau} e^{\beta s} |Y_s|^{p-1} f_s ds + \int_{t\wedge \tau}^{\tau} e^{\beta s} \left(p \al-\beta + \frac{pK^2}{p-1} +\frac{pK^2}{2\eps} \right) |Y_s|^p  ds  \\ \nonumber
&& \quad -  p \int_{t\wedge \tau}^\tau e^{\beta s} |Y_{s-}|^{p-1} \check{Y}_{s-} \left( Z_s dW_s  + dM_s +  \int_\cU \psi_s(u) \tpi(du,ds) \right) \\ \nonumber
& &\quad - \frac{1}{2} \int_{t\wedge \tau}^{\tau} e^{\beta s} \int_\cU \left[ |Y_{s-}+\psi_s(u)|^p -|Y_{s-}|^p - p|Y_{s-}|^{p-1}  \check{Y}_{s-} \psi_s(u) \right] \tpi(du,ds) \\ \nonumber
& &\quad - \frac{1}{2} \int_{t\wedge \tau}^{\tau} e^{\beta s} \int_\cU \left[ |Y_{s-}+\psi_s(u)|^p -|Y_{s-}|^p - p|Y_{s-}|^{p-1}  \check{Y}_{s-} \psi_s(u) \right] \mu(du) ds \\ \nonumber
&& \quad +\frac{p\eps}{2} \int_{t\wedge \tau}^{\tau} e^{\beta s} |Y_{s-}|^{p-2} \1_{Y_{s-}\neq 0}\|\psi_s  \1_{|\psi_s| < \delta}  \|^2_{\bL^2_\mu} ds \\ \label{eq:Lp_apriori_estim_1}
&& \quad + p K\int_{t\wedge \tau}^{\tau} e^{\beta s} |Y_{s-}|^{p-1} \|\psi_s \1_{|\psi_s|\geq \delta} \|_{\bL^1_\mu}ds . 
\end{eqnarray}
Let us explain how to deal with this inequality. 
\begin{itemize}
\item For $\eps > 0$ fixed by Lemma \ref{lem:tech_inequality}, we can take $\beta$ large enough such that 
$$\beta >p \al + \frac{pK^2}{p-1} +\frac{pK^2}{2\eps}$$
and the term  $\int_{0}^{.} e^{\beta s} |Y_s|^p  ds$ can be removed (or put on the left-hand side). Again $\beta$ depends only on $\al$, $K$ and $p$.
\item Using again the fundamental sequence of stopping times $\tau_k$ for the local martingale 
$$\int_{0}^{.} e^{\beta s} |Y_{s-}|^{p-1} \check{Y}_{s-} \left( Z_s dW_s  + dM_s +  \int_\cU \psi_s(u) \tpi(du,ds) \right)$$
and taking $\tau=\tau_k$ and the expectation in \eqref{eq:Lp_apriori_estim_1}, this local martingale term will disappear. 
\item From Lemma 3.67 in \cite{jaco:79} and \eqref{eq:expectation_pi} we deduce
$$0\le \E \int_{0}^{T} e^{\beta s} \int_\cU \left[ |Y_{s-}+\psi_s(u)|^p -|Y_{s-}|^p - p|Y_{s-}|^{p-1}  \check{Y}_{s-} \psi_s(u) \right] \mu(du) ds < +\infty.$$

This implies $\P$-a.s.\ that 
\begin{eqnarray} \nonumber
&& -\frac{1}{2}  \int_{0}^{\tau} e^{\beta s} \int_\cU \left[ |Y_{s-}+\psi_s(u)|^p -|Y_{s-}|^p - p|Y_{s-}|^{p-1}  \check{Y}_{s-} \psi_s(u) \right] \mu(du) ds \\ \nonumber
&& \quad +\frac{p\eps}{2} \int_{0}^{\tau} e^{\beta s} |Y_{s-}|^{p-2}\1_{Y_{s-}\neq 0} \int_\cU |\psi_s(u)|^2  \1_{|\psi_s(u)| < \delta} \mu(du) ds \\ \nonumber
&& \quad +  p K  \int_{0}^{\tau} e^{\beta s} |Y_{s-}|^{p-1}\int_\cU |\psi_s(u)| \1_{|\psi_s(u)|\geq \delta} \mu(du) ds \\ 
&&=\frac{1}{2}\int_{0}^{\tau} e^{\beta s} \int_\cU \left[ \Gamma(Y_{s-},\psi_s(u),K,\eps,p) - \Psi(Y_{s-},\psi_s(u),p) \right] \mu(du) ds, \label{eq:non_pos_int}
\end{eqnarray}
with
\begin{eqnarray} \label{eq:def_Psi}
\Psi(a,b,p) &=& |a+b|^p - |a|^p - p |a|^{p-2} \langle a, b \rangle \1_{a \neq 0} \\ \label{eq:def_Gamma}
\Gamma(a,b,K,\eps,p) & = & 2Kp|a|^{p-1} |b| \1_{|b| \geq \vartheta(\eps,p)|a|} + p\eps |a|^{p-2}|b|^2 \1_{|b| < \vartheta(\eps,p)|a|}.
\end{eqnarray}
Recall that we have chosen $\delta = \vartheta(\eps,p)|Y_{s-}|$ if $Y_{s-}\neq 0$ (and any $\delta > 0$ if $Y_{s-}=0$). From Lemma \ref{lem:tech_inequality}, for any $(\omega,s,u)\in \Omega \times [0,T] \times \cU$, 
$$\Gamma(Y_{s-},\psi_s(u),K,\eps,p) - \Psi(Y_{s-},\psi_s(u),p) \leq 0.$$
Hence the integral \eqref{eq:non_pos_int} is non positive: $\P$-a.s.
$$\int_{0}^{\tau} e^{\beta s} \int_\cU \left[ \Gamma(Y_{s-},\psi_s(u),K,\eps,p) - \Psi(Y_{s-},\psi_s(u),p) \right] \mu(du) ds \leq 0.$$

\item Now in \eqref{eq:Lp_apriori_estim_1} the only uncontrolled remaining term on the right-hand side will be:
\begin{eqnarray*}
- \frac{1}{2} \int_{t\wedge \tau}^{\tau} e^{\beta s} \int_\cU \left[ |Y_{s-}+\psi_s(u)|^p -|Y_{s-}|^p - p|Y_{s-}|^{p-1}  \check{Y}_{s-} \psi_s(u) \right] \tpi(du,ds) 
\end{eqnarray*}
which is a local martingale. Thus it can be cancelled using another fundamental sequence $\hat \tau_k$.
\end{itemize}
Thereby \eqref{eq:Lp_apriori_estim_1} gives for $\tau =\tau_k \wedge \widehat \tau_k$
\begin{eqnarray} \nonumber
&&\E \int_{0}^{\tau} e^{\beta s} \left(\beta - p \al- \frac{pK^2}{p-1} -\frac{pK^2}{2\eps} \right) |Y_s|^p  ds \\ \nonumber
&& + \frac{c(p)}{2}\E \int_{0}^{\tau} e^{\beta s} |Y_{s}|^{p-2} |Z_s|^2 \1_{Y_s\neq 0} ds + c(p) \E \int_{0}^{\tau} e^{\beta s} |Y_{s-}|^{p-2}  \1_{Y_{s-}\neq 0} d[ M ]^c_s \\ \nonumber
&& + \frac{c(p)}{2} \E \int_{0}^{\tau}  \int_\cU e^{\beta s} |\psi_s(u)|^2  \left( |Y_{s-}|^2 \vee  |Y_{s-} +\psi_s(u)|^2 \right)^{p/2-1} \1_{|Y_{s-}|\vee |Y_{s-} + \psi_s(u)| \neq 0}\pi(du,ds) \\  \nonumber
&& + c(p) \E \sum_{0< s \leq \tau} e^{\beta s}  \left( |Y_{s-}|^2 \vee  |Y_{s-} + \Delta M_s|^2 \right)^{p/2-1}  \1_{|Y_{s-}|\vee |Y_{s-} + \Delta M_s| \neq 0} |\Delta M_s|^2\\   \label{eq:Lp_estim_p_leq_2}
&& \leq \E \left( e^{\beta \tau}  |Y_{\tau}|^p\right)  + p \E \int_{0}^{\tau} e^{\beta s} |Y_s|^{p-1} f_s ds .\end{eqnarray}
Recall that $X$ is the quantity
$$X =  e^{\beta T} |\xi|^p +   p \int_0^T e^{\beta s} |Y_s|^{p-1} f_s ds.$$
Then we can pass to the limit on $k$ in \eqref{eq:Lp_estim_p_leq_2}, and we obtain the same estimate for $\tau=T$ and $\E(X)$ on the right-hand side, that is \eqref{eq:Lp_estim_p_leq_2_mart}.

\noindent \textbf{ Step 2:} In this part of the proof we prove that for some constant $\kappa_{p}$ (depending also on $\beta$ and $K$):
$$\E \left( \sup_{t\in [0,T]} e^{\beta t} |Y_t|^p\right) =\E \left( Y_*^{\beta,p} \right) \leq \kappa_{p}\E \left[ X + \int_{0}^{T} e^{\beta s}|Y_{s-}|^{p-1} \|\psi_s\|_{\bL^1_\mu + \bL^2_\mu} ds\right].$$
From \eqref{eq:control_integral_pi} with $t=0$ and $\tau=T$, and from the choice of $\beta$ we have:
\begin{eqnarray*} \nonumber
&& 0 \leq \int_{0}^T e^{\beta s} \int_\cU \left[ |Y_{s-}+\psi_s(u)|^p -|Y_{s-}|^p - p|Y_{s-}|^{p-1}  \check{Y}_{s-} \psi_s(u) \right] \pi(du,ds) \\ \nonumber
&& \quad \leq X  + pK \int_{0}^{T} e^{\beta s}|Y_{s-}|^{p-1} \|\psi_s\|_{\bL^1_\mu + \bL^2_\mu} ds  +  p \sup_{t\in[0,T]} \left( |\Gamma_t| + |\Theta_t | + |\Xi_t| \right) ,
\end{eqnarray*}
where
\begin{eqnarray*}
\Gamma_t & = & \int_0^t e^{\beta s}|Y_s|^{p-1} \check{Y}_s Z_s dW_s ,\\
\Theta_t & = & \int_0^t e^{\beta s}|Y_s|^{p-1} \check{Y}_s dM_s, \quad \Xi_t =  \int_0^t e^{\beta s}|Y_s|^{p-1} \check{Y}_s  \int_\cU \psi_s(u) \tpi(du,ds)  .
\end{eqnarray*}
From Theorem 3.15 in \cite{jaco:79}, taking the expectation in the previous inequality we obtain:
\begin{eqnarray} \nonumber
&& 0 \leq\E \int_{0}^T e^{\beta s} \int_\cU \left[ |Y_{s-}+\psi_s(u)|^p -|Y_{s-}|^p - p|Y_{s-}|^{p-1}  \check{Y}_{s-} \psi_s(u) \right] \mu(du)ds \\ \label{eq:control_loc_mart}
&& \quad \leq \E(X)  + pK \E \int_{0}^{T} e^{\beta s}|Y_{s-}|^{p-1} \|\psi_s\|_{\bL^1_\mu + \bL^2_\mu} ds  +  p \E \left[ \sup_{t\in[0,T]} \left( |\Gamma_t| + |\Theta_t | + |\Xi_t| \right)\right] .
\end{eqnarray}
Coming back to \eqref{eq:Lp_apriori_estim_1}, the last three terms on the right-hand side are non positive (by the same arguments as in Step 1). From the convexity of $|x|^p$, the last local martingale can be controlled by \eqref{eq:control_loc_mart}. Hence from the Burkholder-Davis-Gundy inequality we obtain
\begin{eqnarray*}
\E \left( Y_*^{p,\beta} \right) & \leq & \E\left( X \right) + pK \E \int_{0}^{T} e^{\beta s}|Y_{s-}|^{p-1} \|\psi_s\|_{\bL^1_\mu + \bL^2_\mu} ds + k_p \E \left( [\Gamma]_T^{1/2} + [\Theta]_T^{1/2} + [\Xi]_T^{1/2} \right).
\end{eqnarray*}
The bracket $[\Gamma]_T^{1/2}$ can be handled as in \cite{bria:dely:hu:03}:
\begin{eqnarray*}
k_p \E \left(  [\Gamma]_T^{1/2} \right) & \leq & \frac{1}{6}  \E \left( Y_*^{p,\beta} \right) +  \frac{3k_p^2 }{2} \E \left( \int_0^T e^{\beta s} |Y_{s}|^{p-2} |Z_s|^2 \1_{Y_s\neq 0} ds \right).
\end{eqnarray*}
For the other terms since $p>1$ we have
\begin{eqnarray*}
k_p \E \left(  [\Theta]_T^{1/2} \right) & \leq & k_p \E\left[ \left( \int_0^T e^{2\beta s} \left( |Y_{s-}|^2 \vee  |Y_{s-} + \Delta M_s|^2 \right)^{p-1}  \1_{|Y_{s-}|\vee |Y_{s-} + \Delta M_s| \neq 0} d[M]_s \right)^{1/2} \right] \\
& \leq & k_p \E\left[ \left( \sup_{s \in [0,T]} e^{\beta s}\left( |Y_{s-}|^2 \vee  |Y_{s-} + \Delta M_s|^2 \right)^{p/2}\right)^{1/2} \right. \\
&& \qquad \quad \left. \left( \int_0^T e^{\beta s}\left( |Y_{s-}|^2 \vee  |Y_{s-} + \Delta M_s|^2 \right)^{p/2-1}  \1_{|Y_{s-}|\vee |Y_{s-} + \Delta M_s| \neq 0}d[M]_s \right)^{1/2} \right] \\
& \leq  & \frac{1}{6}  \E \left( Y_*^{p,\beta} \right) +  \frac{3k_p^2 }{2} \E \left( \int_0^T e^{\beta s} |Y_{s-}|^{p-2}  \1_{|Y_{s-}| \neq 0} d[M]^c_s \right. \\
&& \qquad \quad \left. + \sum_{0< s \leq T}  e^{\beta s}\left( |Y_{s-}|^2 \vee  |Y_{s-} + \Delta M_s|^2 \right)^{p/2-1} \1_{|Y_{s-}|\vee |Y_{s-} + \Delta M_s| \neq 0} |\Delta M_s|^2 \right),
\end{eqnarray*}
and by the same argument
\begin{eqnarray*}
&&k_p \E \left(  [\Xi]_T^{1/2} \right) \leq  \frac{1}{6}  \E \left( Y_*^{p,\beta} \right) \\
&&\quad +  \frac{3k_p^2 }{2} \E  \int_{0}^{T} e^{\beta s} \int_\cU |\psi_s(u)|^2  \left( |Y_{s-}|^2 \vee  |Y_{s-} +\psi_s(u)|^2 \right)^{p/2-1} \1_{|Y_{s-}|\vee |Y_{s-} + \psi_s(u)| \neq 0} \pi(du,ds).
\end{eqnarray*}
Using \eqref{eq:Lp_estim_p_leq_2_mart}, we deduce that there exists a constant $\kappa_p$ depending only on $p$ such that
\begin{eqnarray*}
 \E \left( \sup_{t\in [0,T]} e^{\beta t}|Y_t|^p \right)  & \leq & \kappa_p\E \left[ X + \int_{0}^{T} e^{\beta s}|Y_{s-}|^{p-1} \|\psi_s\|_{\bL^1_\mu + \bL^2_\mu} ds\right].
\end{eqnarray*}

\noindent \textbf{ Step 3:} Let us derive now a priori estimates for the martingale part of the BSDE. We use Corollary 1 in \cite{krus:popi:14}: 
\begin{eqnarray} \nonumber
&& \E \left( \int_0^Te^{2\beta s/p}  |Z_s|^2 ds \right)^{p/2}  =  \E \left( \int_0^T e^{2\beta s/p} \1_{Y_s \neq 0} |Z_s|^2 ds\right)^{p/2} \\ \nonumber
&&\quad = \E \left( \int_0^T \left( e^{\beta s/p} \left| Y_{s} \right|\right)^{2-p} e^{\beta s} \left| Y_{s} \right|^{p-2}\1_{Y_s \neq 0}  |Z_s|^2ds \right)^{p/2} \\ \nonumber
&& \quad \leq \E \left[ \left( \sup_{t\in [0,T]} e^{\beta t/p} |Y_{t}|\right)^{p(2-p)/2} \left(\int_0^T e^{\beta s} \left| Y_{s} \right|^{p-2}\1_{Y_s \neq 0}  |Z_s|^2 ds \right)^{p/2}\right]\\ \nonumber
&& \quad \leq \left\{\E \left[ \sup_{t\in [0,T]} e^{\beta t} |Y_{t}|^{p}\right]\right\}^{(2-p)/2}  \left\{\E \int_0^T e^{\beta s} \left| Y_{s} \right|^{p-2} \1_{Y_s \neq 0} |Z_s|^2 ds \right\}^{p/2} \\ \label{eq:trick_control_mart_part}
&& \quad  \leq \frac{2-p}{2} \E \left[ \sup_{t\in [0,T]} e^{\beta t} |Y_{t}|^{p}\right] + \frac{p}{2} \E \int_0^T e^{\beta s} \left| Y_{s} \right|^{p-2}\1_{Y_s \neq 0} |Z_s|^2 ds
\end{eqnarray}
where we have used H\"older's and Young's inequality with $ \frac{2-p}{2} + \frac{p}{2}=1$. With Inequality \eqref{eq:Lp_estim_p_leq_2} we deduce:
\begin{equation*}
\E \left( \int_0^Te^{2\beta s/p} |Z_s|^2 ds \right)^{p/2}  \leq \wtil \kappa_p \E \left[ X + \int_{0}^{T} e^{\beta s}|Y_{s-}|^{p-1} \|\psi_s\|_{\bL^1_\mu + \bL^2_\mu} ds\right].
\end{equation*}
The same argument can be used to control $[M]^c$. For the pure-jump part of $[M]$ we have using the function $u_\eps$ defined in the proof of Lemma 7 in \cite{krus:popi:14}:
\begin{eqnarray*}
&& \E \left( \sum_{0<s\leq T}e^{2\beta s/p} |\Delta M_s|^2 \right)^{p/2} \\
&&\quad = \E \left( \sum_{0<s\leq T} e^{2\beta s/p}\left( u_\eps(|Y_{s-}| \vee  |Y_{s-} + \Delta M_s|) \right)^{2-p}\left(u_\eps( |Y_{s-}| \vee  |Y_{s-} + \Delta M_s| )\right)^{p-2} |\Delta M_s|^2 \right)^{p/2} \\
&& \quad \leq \E \left[ \left( e^{\beta */p}u_\eps(Y_{*})\right)^{p(2-p)/2} \left(\sum_{0<s\leq T} \left( u_\eps(|Y_{s-}| \vee  |Y_{s-} + \Delta M_s|) \right)^{p-2} |\Delta M_s|^2 \right)^{p/2}\right]\\
&& \quad \leq \left\{\E \left[e^{\beta *} \left( u_\eps(Y_{*}) \right)^{p}\right]\right\}^{(2-p)/2} \\
&&\qquad \qquad \times \left\{\E\left(  \sum_{0< s \leq T} \left( u_\eps(|Y_{s-}| \vee  |Y_{s-} + \Delta M_s|) \right)^{p-2} |\Delta M_s|^2 \right) \right\}^{p/2} \\
&& \quad \leq \frac{2-p}{2} \E \left[ e^{\beta *}\left( u_\eps(Y_{*}) \right)^{p}\right] + \frac{p}{2}\E\left(  \sum_{0< s \leq T} \left( u_\eps(|Y_{s-}| \vee  |Y_{s-} + \Delta M_s| )\right)^{p-2} |\Delta M_s|^2 \right).
\end{eqnarray*}
Let $\eps$ go to zero. We use a convergence result, which is a direct consequence of the proof of Lemma 9 in \cite{krus:popi:14} to obtain that
\begin{eqnarray*}
&& \E \left(  \sum_{0<s\leq T}e^{2\beta s/p} |\Delta M_s|^2  \right)^{p/2} \\
&& \quad \leq \frac{2-p}{2} \E \left( e^{\beta *}|Y_{*}|^{p} \right) + \frac{p}{2}\E\left( \sum_{0\leq s < T} e^{\beta s}\left( |Y_{s-}| \vee  |Y_{s-} + \Delta M_s| \right)^{p-2}  \1_{|Y_{s-}|\vee |Y_{s-} + \Delta M_s| \neq 0}|\Delta M_s|^2 \right)\\
&& \quad \leq \wtil \kappa_p \E \left[ X + \int_{0}^{T} e^{\beta s}|Y_{s-}|^{p-1} \|\psi_s\|_{\bL^1_\mu + \bL^2_\mu} ds\right].
\end{eqnarray*}
The same argument shows that 
\begin{eqnarray*}
\E \left( \int_0^Te^{2\beta s/p} \int_\cU |\psi_s(u)|^2 \pi(du,ds) \right)^{p/2} & \leq & \wtil \kappa_p \E \left[ X + \int_{0}^{T} e^{\beta s}|Y_{s-}|^{p-1} \|\psi_s\|_{\bL^1_\mu + \bL^2_\mu} ds\right]. 
\end{eqnarray*}

\noindent \textbf{ Step 4:} Now we prove the wanted estimate. Recall that we have found a constant $\hat \kappa_p$ such that 
 \begin{eqnarray*}
 &&\E\left[  Y^{\beta,p}_* +  \left( \int_0^T e^{2\beta s/p}Z_s^2 ds \right)^{p/2} + \left( \int_0^T e^{2\beta s/p}d[M]_s \right)^{p/2} \right.\\
  &&\qquad \left. 
  + \left( \int_0^T e^{2\beta s/p}\int_\cU |\psi_s(u)|^2 \pi(du,ds) \right)^{p/2}\right] \leq \hat \kappa_p \E \left[ X + \int_{0}^{T} e^{\beta s}|Y_{s-}|^{p-1} \|\psi_s\|_{\bL^1_\mu + \bL^2_\mu} ds\right]
  \end{eqnarray*}
where 
$$X =e^{\beta T}  |\xi|^p +   p \int_0^T e^{\beta s}|Y_s|^{p-1} f_s ds.$$
Using Inequality \eqref{eq:norm_equiv} we know that there exists some constant $K_{p,T}$ such that
 $$\E  \int_{0}^{T} e^{\beta s}\|\psi_s\|^p_{\bL^1_\mu + \bL^2_\mu} ds \leq K_{p,T} \E \left( \int_0^T e^{2\beta s/p}\int_\cU |\psi_s(u)|^2 \pi(du,ds) \right)^{p/2}.$$
Young's inequality leads to: 
 \begin{eqnarray*}
\hat \kappa_p \E \left[ \int_{0}^{T} e^{\beta s}|Y_{s-}|^{p-1} \|\psi_s\|_{\bL^1_\mu + \bL^2_\mu} ds\right] & \leq & (2K_{p,T})^{\frac{1}{p-1}}\frac{(p-1)(\hat \kappa_p)^{\frac{p}{p-1}}}{p^{\frac{p}{p-1}}} \E \int_{0}^{T} e^{\beta s}|Y_{s-}|^{p} ds \\
& + & \frac{1}{2K_{p,T}} \E  \int_{0}^{T} e^{\beta s}\|\psi_s\|^p_{\bL^1_\mu + \bL^2_\mu} ds.
 \end{eqnarray*}
Now from Inequality \eqref{eq:Lp_estim_p_leq_2_mart} 
  \begin{eqnarray*}
 &&\E\left[  Y^{\beta,p}_* +  \left( \int_0^T e^{2\beta s/p}Z_s^2 ds \right)^{p/2} + \left( \int_0^T e^{2\beta s/p}d[M]_s \right)^{p/2} \right.\\
  &&\qquad \left. 
  +\frac{1}{2} \left( \int_0^T e^{2\beta s/p}\int_\cU |\psi_s(u)|^2 \pi(du,ds) \right)^{p/2}\right] \leq \widetilde{C_{p}} \ \E (X )
  \end{eqnarray*}
  with
$$\widetilde{C_{p}}  =   \hat \kappa_p + \kappa_{p,\beta}  (2K_{p,T})^{\frac{1}{p-1}}\frac{(p-1)(\hat \kappa_p)^{\frac{p}{p-1}}}{p^{\frac{p}{p-1}}} .$$
The key point is that $\widetilde{C_{p}}$ depends on $p$, $T$ and the regularity constants of the generator $f$. Then 
 \begin{eqnarray*}
p \widetilde{C_{p}} \E \int_0^T e^{\beta s}|Y_s|^{p-1} f_s ds & \leq & p \widetilde{C_{p}}\left( e^{\beta \frac{p-1}{p}* }|Y_*|^{p-1} \right)\int_0^T e^{\beta s /p}f_s ds\\
 & \leq & \frac{1}{2} \E \left( e^{\beta *} |Y_*|^p\right) + d_p \left(  \int_0^T e^{\beta s/p} f_s ds\right)^p.
 \end{eqnarray*}
Therefore we have proved that for any $\beta$ large enough (with a lower bound depending only on $\al$, $K$ and $p$)
\begin{eqnarray*}
&&\E\left[ \left(  \sup_{t\in [0,T]} e^{\beta t} |Y_t|^p \right)+  \left( \int_0^T e^{2\beta s/p} Z_s^2 ds \right)^{p/2}  + \left( \int_0^T \int_\cU e^{2\beta s/p} |\psi_s(u)|^2 \pi(du,ds) \right)^{p/2}\right. \\
&& \qquad \left.+ \left( \int_0^T e^{2\beta s/p} d[M]_s\right)^{p/2} \right] \leq C \E \left[ e^{\beta T} |\xi|^p+ \left( \int_0^T e^{\beta r/p} f_r dr \right)^p\right]
\end{eqnarray*}
where $C$ just depends on $p$. This gives the desired estimate. 
\end{proof}

\vspace{0.5cm}

\noindent \textbf{ Theorem 2} 
\textit{
Under Assumptions $\mathbf{(H_{ex})}$ and \eqref{eq:int_cond}, there exists a unique solution $(Y,Z,\psi,M)$ in $\cE^p(0,T)$ to the BSDE \eqref{eq:gene_BSDE}. Moreover for some constant $C=C_{p,K,T}$}
\begin{eqnarray*}
&& \E \left[\sup_{t\in[0,T]}   |Y_t|^p +  \left( \int_0^T  |Z_t|^2 dt \right)^{p/2} +  \left(  \int_0^T  \int_{\cU} |\psi_s(u)|^2 \pi(du, ds) \right)^{p/2} + \left(  [ M ]_T \right)^{p/2}\right] \\
&& \qquad \qquad \leq C \E \left[|\xi|^p  + \left( \int_0^T |f(r,0,0,0)| dr \right)^p \right].
\end{eqnarray*}
\begin{proof}
We can follow the proof of Theorem 2 in \cite{krus:popi:14}. If we define 
$$\xi_n = q_n(\xi), \qquad f_n(t,y,z,\psi) = f(t,y,z,\psi) - f(t,0,0,0) + q_n(f(t,0,0,0)),$$
with $q_n(x) = xn/(|x|\vee n)$, thanks to Theorem 1 in \cite{krus:popi:14}, we have a unique solution $(Y^n,Z^n,\psi^n,M^n)$ in $\cE^2$, and thus in $\cE^p$ for any $p> 1$. From $\mathbf{(H_{ex})}$ it can be proved as in \cite{krus:popi:14} that Condition \textbf{(C)} holds:
\begin{eqnarray*}
&& \langle \frac{(Y^m_t-Y^n_t)}{|Y^m_t-Y^n_t|} \1_{Y^m_t-Y^n_t\neq 0},f_{m}(t,Y^{m}_t,Z^{m}_t,\psi^{m}_t) - f_{n}(t,Y^{n}_t,Z^{n}_t,\psi^{n}_t) \rangle \\
&&\quad  \leq |q_m(f(t,0,0,0)) - q_n(f(t,0,0,0))|+ K|Z^m_t-Z^n_t| + K \|\psi^m_t-\psi^n_t\|_{\bL^1_\mu + \bL^2_\mu} .
\end{eqnarray*}
Proposition 3 shows that
\begin{eqnarray*}
&&\E \left[\sup_{t\in[0,T]}  |Y^{m}_t-Y^n_t|^p + \left( \int_0^T |Z^{m}_s-Z^n_s|^2 ds \right)^{p/2}+ \left( [ M^{m}-M^n  ]_T \right)^{p/2} \right. \\
&&\quad  \left.+  \left(  \int_0^T  \int_{\cU} |\psi^{m}_s(u)-\psi^n_s(u)|^2 \pi(du, ds) \right)^{p/2} 
\right] \\
&& \qquad \qquad \leq C \E \left[  |\xi_{m}-\xi_n|^p + \left( \int_0^T |q_{m}(f(r,0,0,0))-q_{n}(f(r,0,0,0))| dr \right)^p \right].
\end{eqnarray*}
Thus $(Y^n,Z^n,\psi^n,M^n)$ is a Cauchy sequence in $\cE^p(0,T)$ and the conclusion follows.
\end{proof}

Again the Bichteler-Jacod inequality \eqref{eq:BJ_ineq} implies that the sequence $(\psi^n)$ is also a Cauchy sequence in $\bL^1_\mu + \bL^2_\mu$ (or in $\bL^p_\mu + \bL^2_\mu$) and the limit $\psi$ belongs to these Banach spaces.

\section{Comparison principle and extension to random terminal time (Theorem 3)}
%------------

\subsection{Comparison principle} \label{ssect:comp}
%--------------

The comparison principle (Proposition 4 in \cite{krus:popi:14}) holds true under Condition $(\mathbf{ H_{comp}})$, which reinforces Assumption (H3). Now for $p\in (1,2)$ we assume that
\begin{enumerate}
\item[(H3')] $f$ is Lipschitz continuous w.r.t.\ $z$ with constant $K$ and for each $(y,z,\psi,\phi) \in \R\times \R^k \times (\bL^1_\mu+\bL^2_\mu)^2$, there exists a predictable process $\kappa = \kappa^{y,z,\psi,\phi} : \Omega \times [0,T] \times \cU \to \R$ such that:
$$f(t,y,z,\psi) - f(t,y,z,\phi) \leq  \int_\cU (\psi(u)-\phi(u)) \kappa^{y,z,\psi,\phi}_t(u) \mu(du)$$
with $\P\otimes \mbox{Leb} \otimes \mu$-a.e.  for any $(y,z,\psi,\psi')$,
\begin{itemize}
\item $-1 \leq \kappa^{y,z,\psi,\phi}_t(u)$
\item $|\kappa^{y,z,\psi,\phi}_t(u)|\leq \ell(u)$, where $\ell$ belongs to $\bL^\infty_\mu \cap \bL^2_\mu$. \end{itemize}
\end{enumerate}
We say that $(\mathbf{ H_{comp}})$ is satisfied if (H1)--(H2) and (H3') hold.
\begin{Lemma} \label{lem:hcomp->hex}
Assumption $(\mathbf{ H_{comp}})$ implies Condition $\mathbf{ (H_{ex})}$, that is $f$ is Lipschitz continuous w.r.t. $\psi$. 
\end{Lemma}
\noindent \begin{proof}
Indeed for $p<2$, we have to take $\psi$ and $\phi$ in $\bL^1_\mu + \bL^2_\mu$. Thus if $\ell$ belongs to $\bL^\infty_\mu \cap \bL^2_\mu$, the dual space of $\bL^1_\mu + \bL^2_\mu$ (see \cite{krei:petu:seme:82}, Chapter 3, Theorem 3.1), then for $\psi$ and $\phi$ in $\bL^1_\mu + \bL^2_\mu$, we obtain:
$$|f(t,y,z,\psi) - f(t,y,z,\phi) | \leq \|\ell\|_{\bL^\infty_\mu \cap \bL^2_\mu}  \|\psi-\phi\|_{\bL^1_\mu + \bL^2_\mu}.$$
\end{proof}

Then under $(\mathbf{ H_{comp}})$, the proof of Proposition 4 in \cite{krus:popi:14} remains exactly the same.

\subsection{Random terminal time}
%--------------

Now we assume that $\tau$ is a stopping time for the filtration $\mathbb F$, which need not be bounded (as in Section 6 of \cite{krus:popi:14}). We want to solve the following BSDE: $\P$-a.s., for all $0\leq t \leq T$,
\begin{eqnarray} \nonumber
Y_{t\wedge \tau} & = & Y_{T\wedge \tau} + \int_{t\wedge \tau}^{T\wedge \tau} f(s,Y_s, Z_s,\psi_s) ds -\int_{t\wedge \tau}^{T\wedge \tau} Z_sdW_s \\ \label{eq:gene_BSDE_rand_time} 
& -& \int_{t\wedge \tau}^{T\wedge \tau} \int_\cU \psi_s(u) \tpi(du,ds) - \int_{t\wedge \tau}^{T\wedge \tau} dM_s
\end{eqnarray}
with the condition that $\P$-a.s. on the set $\{t\geq \tau \}$, $Y_t =\xi$ and $Z_t =\psi_t=M_t=0$. Note that this equation was denoted (36) in \cite{krus:popi:14}.

On the generator, Assumptions  $\mathbf{ (H_{ex})}$ still hold with a monotonicity constant $\al$ and a Lipschitz constant $K$, but the growth condition (H2) is replaced by:
\begin{equation}\tag{H2''}
\forall r > 0,\ \forall n \in \N,\quad \sup_{|y|\leq r} (|f(t,y,0,0)-f(t,0,0,0)|) \in L^1(\Omega \times (0,n)).
\end{equation}
and the condition \eqref{eq:int_cond} is replaced by 
\begin{equation} \label{eq:int_cond_random_time}
\E \left[ e^{p\rho \tau} |\xi|^p + \int_0^\tau e^{p\rho t} |f(t,0,0,0)|^p dt \right] < +\infty
\end{equation}
(denoted (H5') in \cite{krus:popi:14}, Section 6). We suppose that the constant $\rho$ in \eqref{eq:int_cond_random_time} satisfies
$$\rho > \nu = \alpha + \frac{K^2}{p-1} + \frac{K^2}{2\eps},$$
where the constant $0 < \eps < \frac{p-1}{2}$ is given by Lemma \ref{lem:tech_inequality} and depends only on $K$ and $p$ (see \eqref{eq:estim_eps} for a value of $\eps$). As in \cite{krus:popi:14} we suppose that Condition (H6) holds, that is $\xi$ is $\F_\tau$-measurable and 
\begin{equation*}%\tag{H6}
\E\left[ \int_0^\tau e^{p\rho t}|f(t, \xi_t ,\eta_t, \gamma_t)|^p dt \right] < +\infty,
\end{equation*}
where $ \xi_t = \E ( \xi|\F_t) $ and $( \eta, \gamma, N)$ are given by the martingale representation:
$$ \xi = \E ( \xi)  + \int_0^\infty  \eta_s dW_s + \int_0^\infty \int_\cU  \gamma_s(u) \tpi(du,ds) +  N_\tau$$
with
$$\E \left[ \left( \int_0^\infty | \eta_s|^2 ds +  \int_0^\infty \int_\cU | \gamma_s(u)|^2 \pi(du,ds) +[N]_\tau \right)^{p/2} \right] < +\infty.$$

\vspace{0.5cm}
\noindent \textbf{ Proposition 5} 
%\begin{Prop} \label{prop:uniq_random_time}
\textit{
Under conditions \textrm{(H1), (H2''), (H3), \eqref{eq:int_cond_random_time} and (H6)}, the BSDE \eqref{eq:gene_BSDE_rand_time} has at most one solution satisfying}
\begin{eqnarray} \nonumber
&&\E \left[ e^{p\rho (t\wedge \tau)} |Y_{t\wedge \tau}|^p + \int_{0}^{T\wedge \tau}  e^{p\rho s} |Y_{s}|^{p}  ds +  \int_{0}^{T\wedge \tau}  e^{p\rho s} |Y_{s}|^{p-2} |Z_s|^2 \1_{Y_s\neq 0} ds \right] \\ \nonumber
&& + \E \left[\int_{0}^{T\wedge \tau} e^{p\rho s}  |Y_{s-}|^{p-2}  \1_{Y_{s-}\neq 0} d[ M]^c_s \right] \\  \nonumber
&& + \E \left[   \int_{t\wedge \tau}^{T\wedge \tau}\int_\cU e^{p\rho s}  \left( |Y_{s-}|^2 \vee  |Y_{s-} + \psi_s(u)|^2 \right)^{p/2-1} \1_{|Y_{s-}| \vee  |Y_{s-} + \psi_s(u)| \neq 0}| \psi_s(u)|^2\pi(du,ds)\right] \\ \label{eq:rnd_term_time_apriori_estim}
&& + \E\left[  \sum_{0 < s \leq T\wedge  \tau}e^{p\rho s} |\Delta M_s|^2  \left( |Y_{s-}|^2 \vee  |Y_{s-} + \Delta M_s|^2 \right)^{p/2-1} \1_{|Y_{s-}| \vee  |Y_{s-} + \Delta M_s| \neq 0} \right]  
 < +\infty. 
\end{eqnarray}
%\end{Prop}
\begin{proof}
From the assumption on $f$, Young's inequality and Lemma \ref{lem:tech_inequality}, we choose $\eps > 0$ and $\delta = \vartheta(\eps,p)|\what y|$ such that 
\begin{eqnarray*}
&& p|\what y|^{p-1} \check{\what y} (f(s,y,z,\psi)-f(s,y',z',\psi') -\rho|\what y|^p \leq p \left(\al  +\frac{K^2}{p-1} +\frac{K^2}{2\eps} - \rho \right) |\what y|^p \\
&& \qquad + \frac{c(p)}{2}  |\what y|^{p-2} \1_{\what y\neq 0}  |\what z|^2   +\frac{p\eps}{2}  |\what y|^{p-2} \1_{\what y\neq 0} \|\what \psi \1_{|\what \psi | < \delta} \|^2_{\bL^2_\mu} + pK |\what y|^{p-1}  \|\what \psi \1_{|\what \psi | \geq \delta} \|_{\bL^1_\mu}  \\
&& \quad \leq  \frac{c(p)}{2}  |\what y|^{p-2} \1_{\what y\neq 0}  |\what z|^2   +\frac{p\eps}{2}  |\what y|^{p-2} \1_{\what y\neq 0} \|\what \psi \1_{|\what \psi | < \delta} \|^2_{\bL^2_\mu} + pK |\what y|^{p-1}  \|\what \psi \1_{|\what \psi | \geq \delta} \|_{\bL^1_\mu} \\
&& \quad   \leq  \frac{c(p)}{2}  |\what y|^{p-2} \1_{\what y\neq 0}  |\what z|^2   + \frac{1}{2} \Gamma(\what y,\what \psi,K,\eps,p) 
\end{eqnarray*}
where $\Gamma$ is defined by \eqref{eq:def_Gamma}. Then It\^o's formula and the previous inequality give for $0 \leq t \leq T$
\begin{eqnarray} \nonumber
&&e^{p\rho (t\wedge \tau)} |\what Y_{t\wedge \tau}|^p 
 \leq e^{p\rho (T\wedge \tau)}|\what Y_{T\wedge \tau}|^p -  p \int_{t\wedge \tau}^{T\wedge \tau} e^{p\rho s} |\what Y_s|^{p-1} \check{\what Y}_s \what Z_s dW_s   \\ \nonumber
& &\quad -  p  \int_{t\wedge \tau}^{T\wedge \tau} e^{p\rho s}|\what Y_{s-}|^{p-1} \check{\what Y}_{s-} d\what M_s -  p \int_{t\wedge \tau}^{T\wedge \tau}e^{p\rho s} |\what Y_{s-}|^{p-1} \check{\what Y}_{s-}  \int_\cU \what \psi_s(u) \tpi(du,ds) \\ \nonumber
&& \quad - \frac{1}{2} \int_{t\wedge \tau}^{T\wedge \tau} e^{p \rho s} \int_\cU \left[ |Y_{s-}+\psi_s(u)|^p -|Y_{s-}|^p - p|Y_{s-}|^{p-1}  \check{Y}_{s-} \psi_s(u) \right] \tpi(du,ds) \\ \label{eq:Ito_p_tau_estimate}
&& \quad + \frac{1}{2}  \int_{t\wedge \tau}^{T\wedge \tau} e^{p \rho s} \left[ \Gamma(\what Y_{s-},\what \psi_s(u),K,\eps,p) -\Psi(\what Y_{s-},\what \psi_s(u),p)\right]\mu(du)ds 
\end{eqnarray}
$\Psi$ being defined by \eqref{eq:def_Psi}. From Lemma \ref{lem:tech_inequality} the last term is non positive. From the integrability conditions on the solution taking the expectation in \eqref{eq:Ito_p_tau_estimate} leads to 
$$\E \left( e^{p\rho (t\wedge \tau)} |\what Y_{t\wedge \tau}|^p \right) \leq \E \left( e^{p\rho (T\wedge \tau)}|\what Y_{T\wedge \tau}|^p \right) .$$
If we replace $\rho$ by $\rho'$ with $\al  +\frac{K^2}{p-1} < \rho'<\rho$ we obtain the same result, and thus we get for any $0\leq t \leq T $
$$\E \left( e^{p\rho' (t\wedge \tau)} |\what Y_{t\wedge \tau}|^p \right) \leq e^{p(\rho'-\rho)T} \E \left( e^{p\rho (T\wedge \tau)}|\what Y_{T\wedge \tau}|^p \right) .$$
We let $T$ go to infinity to obtain $\what Y_t = 0$. 
Therefore $(Y,Z,\psi,M)$ and $(Y',Z',\psi',M')$ satisty BSDE \eqref{eq:gene_BSDE_rand_time} and $Y=Y'$. Thus we have the same martingale parts and by orthogonality, $\what Z=\what \psi =\what M=0$. Uniqueness of the solution is proved. 

\end{proof}

\vspace{0.5cm}
\noindent \textbf{ Proposition 6} 
%\begin{Prop}  \label{prop:exis_Lp_sol_rnd_time}
\textit{
Under conditions \textrm{(H1), (H2''), (H3), \eqref{eq:int_cond_random_time} and (H6)}, the BSDE \eqref{eq:gene_BSDE_rand_time} has a solution satisfying \eqref{eq:rnd_term_time_apriori_estim}, the right-hand side of the inequality is given by 
$$C \E\left[  e^{p\rho \tau} |\xi|^p  +  \int_{0}^{\tau} e^{p\rho s} |f(s,0,0,0)|^p ds \right].$$
Moreover }
\begin{eqnarray} \nonumber
&& \E \left( \int_{0}^{\tau} e^{2\rho s} |Z_s|^2 ds \right)^{p/2}+ \E\left(  \int_{0}^{ \tau} e^{2\rho s}\int_\cU |\psi_s(u)|^2 \pi(du,ds)  \right)^{p/2}+ \E \left( \int_{0}^{ \tau} e^{2\rho s} d[M]_s\right)^{p/2}\\ \label{eq:rnd_term_time_apriori_estim_3}
&& \qquad  \leq C \E\left[  e^{p\rho \tau} |\xi|^p  +  \int_{0}^{\tau} e^{p\rho s} |f(s,0,0,0)|^p ds \right]. 
\end{eqnarray}
\textit{The constant $C$ depends only on $p$, $K$ and $\alpha$. }\\*
%\end{Prop}
\begin{proof}
For each $n\in \N$ we construct a solution $\{(Y^n,Z^n,\psi^n,M^n), \ t\geq 0\}$, first on the interval $[0,n]$ using Theorem 2:
\begin{eqnarray} \nonumber
Y^n_{t} & = & \E (\xi | \F_n)  + \int_{t}^{n} \1_{[0,\tau]}(s) f(s,Y^n_s, Z^n_s,\psi^n_s) ds -\int_{t}^{n} Z^n_sdW_s - \int_{t}^{n} \int_\cU \psi^n_s(u) \tpi(du,ds) - \int_{t}^{n} dM^n_s.
\end{eqnarray}
And for $t \geq n$ (Assumption (H6)):
$$Y^n_t = \xi_t ,\quad Z^n_t = \eta_t, \quad \psi^n_t(u) = \gamma_t(u), \quad M^n_t = N_t.$$

\begin{itemize}
\item \textbf{ Step 1:} \textit{a priori estimate.} 
\end{itemize}
Again with Young's inequality and for some $\delta > 0$ sufficiently small and any $\eta > 0$
\begin{eqnarray} \nonumber
&& |y|^{p-1}\check{y} f(t,y,z,\psi) \leq \left( \al +\delta +\frac{K^2}{((p-1)-2\delta)} +\frac{K^2}{\eps}\right) |y|^p \\ \nonumber
&& \qquad + \left( \frac{(p-1)}{2} -\delta \right) |y|^{p-2} \1_{y\neq 0}  |z|^2  + \frac{1}{p} |f(t,0,0,0)|^p \left( \frac{p\delta }{p-1}\right)^{1-p} \\ \label{eq:estim_12_10_2014_2}
&& \qquad + \eps |y|^{p-2} \1_{y\neq 0} \|\psi \1_{|\psi|\leq \eta} \|^2_{ \bL^2_\mu} +  K |y|^{p-1}  \|\psi \1_{|\psi|\geq \eta} \|_{ \bL^1_\mu} .
\end{eqnarray}
We choose $\delta > 0$ such that $ \al +2\delta +\frac{K^2}{(p-1-2\delta)} +\frac{K^2}{\eps}\leq \rho$. As in \cite{krus:popi:14}, It\^o's formula for $0 \leq t \leq T \leq n$ and arguments used in the proof of Propositions 3 or 5 give:
\begin{eqnarray} \nonumber
&&\E \left[ e^{p\rho (t\wedge \tau)} |Y^n_{t\wedge \tau}|^p +p\delta   \int_{0}^{T\wedge \tau}  e^{p\rho s} |Y^n_{s}|^{p}  ds\right] \\ \nonumber
&& +p\delta \E \left[  \int_{0}^{T\wedge \tau}  e^{p\rho s} |Y^n_{s}|^{p-2} |Z^n_s|^2 \1_{Y^n_s\neq 0} ds \right]  +c(p)\E  \int_{0}^{T\wedge \tau} e^{p\rho s}  |Y^n_{s-}|^{p-2}  \1_{Y^n_{s-}\neq 0} d[ M^n ]^c_s\\ \nonumber
&& + c(p) \E\left[  \sum_{0 < s \leq T\wedge  \tau}e^{p\rho s} |\Delta M^n_s|^2  \left( |Y^n_{s-}|^2 \vee  |Y^n_{s-} + \Delta M^n_s|^2 \right)^{p/2-1} \1_{|Y^n_{s-}| \vee  |Y^n_{s-} + \Delta M^n_s| \neq 0} \right] \\ \nonumber
&& + \frac{c(p)}{2} \E \int_{0}^{T\wedge \tau}\int_\cU e^{p\rho s}|\psi^n_s(u)|^2  \left( |Y^n_{s-}|^2 \vee  |Y^n_{s-} + \psi^n_s(u)|^2 \right)^{p/2-1} \1_{|Y^n_{s-}| \vee  |Y^n_{s-} + \psi^n_s(u)|\neq 0}\pi(du,ds) \\ 
\label{eq:estim_12_10_2014_3}
&& \leq \E\left[  e^{p\rho(T\wedge \tau)} |Y^n_{T\wedge \tau}|^p  + \left( \frac{p\delta }{p-1}\right)^{1-p}  \int_{0}^{T\wedge \tau} e^{p\rho s} |f(s,0,0,0)|^p ds \right].
\end{eqnarray}

\begin{itemize}
\item \noindent \textbf{Step 2:} \textit{the sequence $(Y^n)$ converges. }
\end{itemize}
Take $m > n$ and define
$$\what Y_t = Y^m_t - Y^n_t, \quad \what Z_t = Z^m_t - Z^n_t, \quad \what \psi_t = \psi^m_t - \psi^n_t, \quad \what M_t = M^m_t - M^n_t.$$
The argument already used to control the generator (see \eqref{eq:estim_12_10_2014_2}) and suitable
modifications (as in the proof of Proposition 3 again) imply that Inequality (43) for $n \leq t \leq m$ in \cite{krus:popi:14} becomes now:
\begin{eqnarray}  \nonumber
&&\E \left[ \sup_{t\geq n} e^{p\rho (t\wedge \tau)} |\what Y_{t\wedge \tau}|^p + \int_{n\wedge \tau}^{m\wedge \tau} e^{p\rho s} |\what Y_s|^p ds \right] 
 \leq  C\E  \int_{n\wedge \tau}^{\tau} e^{p\rho s}|f(s,\xi_s,\eta_s,\gamma_s)|^p ds .
\end{eqnarray}
From the same argument as in the proof of Proposition 5 for $t\leq n$
\begin{eqnarray*}
\E \left(  e^{p\rho (t\wedge \tau)} |\what Y_{t\wedge \tau}|^p \right) +\E  \int_0^{\tau} e^{p\rho s} |\what Y_s|^p  ds &  \leq & \E e^{p\rho (n\wedge \tau)}|\what Y_{n}|^p \leq   C\E  \int_{n\wedge \tau}^{\tau} e^{p\rho s}|f(s,\xi_s,\eta_s,\gamma_s)|^p ds .
\end{eqnarray*}
The convergence of the sequence $Y^n$ follows.

\begin{itemize}
\item \noindent \textbf{Step 3:} \textit{convergence of the martingale part $(Z^n ,\psi^n,M^n)$.} 
\end{itemize}
For the convergence of $(Z^n,M^n)$ the arguments are the same. But for $\psi^n$, we control only 
$$\E \left[ \left( \int_0^{\tau} \int_\cU e^{p\rho s} |\psi^m_s(u)-\psi^n_s(u)|^2  \pi(du,ds) \right)^{p/2} \right].$$
Following the same sketch as in the proof of uniqueness we deduce  
\begin{eqnarray*}  \nonumber
&&\E \left[  \int_{0}^{m\wedge \tau} e^{p\rho s}  |\what Y_{s-}|^{p-2}  \1_{\what Y_{s-}\neq 0} d[ \what M ]^c_s+  \int_{0}^{m\wedge \tau}  e^{p\rho s} |\what Y_{s}|^{p-2} |\what Z_s|^2 \1_{\what Y_s\neq 0} ds\right] \\  \nonumber
&& + \E \left[ \int_{0}^{ m \wedge \tau}\int_\cU e^{p\rho s}|\what \psi_s(u)|^2  \left( |\what Y_{s-}|^2 \vee  |\what Y_{s^-} + \what \psi_s(u)|^2 \right)^{p/2-1} \1_{ |\what Y_{s-}| \vee  |\what Y_{s-} + \what \psi_s(u)| \neq 0}\pi(du,ds) \right] \\ 
&& + \E \left[  \sum_{0 < s \leq m\wedge \tau}e^{p\rho s} |\Delta \what M_s|^2  \left( |\what Y_{s-}|^2 \vee  |\what Y_{s-} + \Delta \what M_s|^2 \right)^{p/2-1} \1_{|\what Y_{s-}| \vee  |\what Y_{s-} + \Delta \what M_s| \neq 0} \right] \\ 
&&\qquad  \leq  C\E  \int_{n\wedge \tau}^{\tau} e^{p\rho s}|f(s,\xi_s,\eta_s,\gamma_s)|^p ds .
\end{eqnarray*}
Then we can use again the argument \eqref{eq:trick_control_mart_part} in order to have a Cauchy sequence for the norm:
$$\E \left( \int_{0}^{\tau} e^{2\rho s} |\what Z_s|^2 ds \right)^{p/2}+ \E\left(  \int_{0}^{ \tau} e^{2\rho s}\int_\cU |\what \psi_s(u)|^2 \pi(du,ds)  \right)^{p/2}+ \E \left( \int_{0}^{ \tau} e^{2\rho s} d[\what  M]_s\right)^{p/2}.$$
Hence it converges to $(Z,\psi,M)$ and from the two previous steps the limit $(Y,Z,\psi,M)$ is a solution of the BSDE \eqref{eq:gene_BSDE_rand_time} which satisfies \eqref{eq:rnd_term_time_apriori_estim} and \eqref{eq:rnd_term_time_apriori_estim_3}.

\end{proof}

From the two previous propositions we deduce the following existence and uniqueness result.

\vspace{0.5cm}
\noindent \textbf{ Theorem 3} 
%\begin{Thm} \label{thm:exis_Lp_sol_rnd_time}
\textit{
Under conditions (H1), (H2''), (H3), \eqref{eq:int_cond_random_time} and (H6), the BSDE \eqref{eq:gene_BSDE_rand_time} has a unique solution satisfying \eqref{eq:rnd_term_time_apriori_estim} and \eqref{eq:rnd_term_time_apriori_estim_3}.}\\*
%\end{Thm}

\section{Technical results}
%-------------

To prove our results in the previous section we used technical Lemmas \ref{lem:sum_Lp_spaces} and \ref{lem:tech_inequality}. Here we give the proof of these results

\noindent \begin{proof} \textbf{ (of Lemma \ref{lem:sum_Lp_spaces}).}
If $\phi^1=\phi \1_{|\phi| \leq \delta} \in \bL^2_\mu$ and $\phi^2=\phi \1_{|\phi| > \delta} \in \bL^p_\mu$, then $\phi = \phi^1+\phi^2$ and the result is trivial. Conversely if $\|\phi \|_{\bL^2_\mu+\bL^p_\mu} < +\infty$, then for any $\eps > 0$, there exists $\phi^1 \in \bL^2_\mu$ and $\phi^2  \in \bL^p_\mu$ with $\phi^1+\phi^2 = \phi$ and
$$\|\phi^1\|_{\bL^2_\mu} +  \|\phi^2\|_{\bL^p_\mu} \leq \|\phi \|_{\bL^2_\mu+\bL^p_\mu} + \eps.$$ 
Now for all $\delta>0$ it holds that
$$|\phi| \1_{|\phi| > \delta}  \leq |\phi^2| + |\phi^1| \1_{|\phi^1| \geq \delta/2} +  |\phi^1| \1_{|\phi^1| < \delta/2}\1_{|\phi| > \delta} .$$
We already know that $\phi^2 \in \bL^p_\mu$ and that $\phi^1\in \bL^2_\mu$. Since $p<2$ it follows that the second term is in $\bL^p_\mu$:
$$ |\phi^1|^p \1_{|\phi^1| \geq \delta/2} \leq \left( \frac{\delta}{2} \right)^{p-2} |\phi^1|^2 \1_{|\phi^1| \geq \delta/2} \leq \left( \frac{\delta}{2} \right)^{p-2} |\phi^1|^2.$$
For the third one, observe that if $|\phi^1| < \delta/2$ and $|\phi^1+\phi^2|=|\phi| > \delta$, then $|\phi^2| \geq \delta/2$. Thus 
$$|\phi^1| \1_{|\phi^1| < \delta/2}\1_{|\phi| > \delta} \leq |\phi^2| \1_{|\phi^1| < \delta/2}\1_{|\phi| > \delta} \leq |\phi^2|.$$
Thus $|\phi| \1_{|\phi| > \delta} \in \bL^p_\mu$. 

Let us now turn to $|\phi| \1_{|\phi| \leq \delta} $. We decompose this term as follows
$$|\phi| \1_{|\phi| \leq \delta}  \leq |\phi^1| + |\phi^2| \1_{|\phi^2| \leq 2\delta} +  |\phi^2| \1_{|\phi^2| > 2\delta}\1_{|\phi| \leq  \delta} .$$
Again we already know that $\phi^2 \in \bL^p_\mu$ and that $\phi^1\in \bL^2_\mu$. Thus the second term is in $\bL^2_\mu$, since for $p<2$:
$$ |\phi^2|^2 \1_{|\phi^2| \leq 2\delta} \leq (2\delta)^{2-p} |\phi^2|^p \1_{|\phi^2| \leq 2\delta}\leq (2\delta)^{2-p}  |\phi^2|^p.$$
For the third one, observe that if $|\phi^2| > 2\delta$ and $|\phi^1+\phi^2| \leq \delta$, then $|\phi^1|> \delta$ and
$$|\phi^2|\le |\phi^2+\phi^1|+|\phi^1|\le \delta+|\phi^1|\le 2|\phi^1|.$$
Thus $|\phi| \1_{|\phi| \leq \delta} \in \bL^2_\mu$.

Finally, if $|\phi| \1_{|\phi| > \delta} \in \bL^p_\mu$, we also have $|\phi| \1_{|\phi| > \delta} \in \bL^1_\mu$, and the conclusion follows.
\end{proof}

\vspace{1cm}
Recall that for $p\in (1,2)$, $K\geq 0$, $\eps > 0$ and $(a,b) \in (\R^d)^2$, we have defined $\Psi$ by \eqref{eq:def_Psi} and $\Gamma$ by \eqref{eq:def_Gamma} as follows:
\begin{eqnarray*}
\Psi(a,b,p) & = & |a+b|^p - |a|^p - p |a|^{p-1} \langle \check a, b \rangle =|a+b|^p - |a|^p - p |a|^{p-2} \langle a, b \rangle \1_{a \neq 0} \\
\Gamma(a,b,K,\eps,p) & = & 2Kp|a|^{p-1} |b| \1_{|b| \geq \vartheta(\eps,p)|a|} + p\eps |a|^{p-2}|b|^2 \1_{|b| < \vartheta(\eps,p)|a|}
\end{eqnarray*}
where
$$\vartheta(\eps,p) = \sqrt{\frac{1}{2}\left(\frac{p-1}{2\eps} \right)^{\frac{2}{2-p}} + \frac{1}{2}} -1.$$
\begin{Lemma} \label{lem:tech_inequality}
Let $K \geq 0$ and let $p\in (1,2)$. Then there exists $0 < \eps < \frac{p-1}{2}$ such that 
$$\forall (a,b) \in (\R^d)^2, \qquad \Psi(a,b,p) \geq \Gamma(a,b,K,\eps,p).$$
Let us emphasize that $\eps$ depends on $K$ and $p$.
\end{Lemma}
\begin{proof}
First observe that that for $a=0$ the inequality holds for all $\eps>0$ and $b\in \R^d$. Assume in the sequel that $a\neq  0$. For $t\in \R$, $\tau^2\in [0,\infty)$ and $\epsilon \in (0,\infty)$ let 
$$\psi(t,\tau^2,p)=\left( ((1+t)^2 + \tau^2 \right)^{p/2} - 1 - p t $$
and
$$\gamma(t,\tau^2,K,\epsilon,p)= 2Kp \left( t^2+\tau^2 \right)^{1/2} \1_{(|t|^2+\tau^2)^{1/2} \geq \vartheta(\eps,p)}
 +  p\eps (t^2+\tau^2)  \1_{(|t|^2+\tau^2)^{1/2}  < \vartheta(\eps,p)}.$$
For all $b \in \R^d$ there exist a unique $t \in \R$ and a unique $c\in \R^d$ with $\langle a, c \rangle= 0$ and $b=ta+c$. If we choose $t\in \R$ and $c\in \R^d$ in this way and let $\tau^2 = \frac{|c|^2}{|a|^2} \geq 0$, we obtain that
\begin{eqnarray*}
\Psi(a,b,p) & = & |a+b|^p - |a|^p - p |a|^{p-2} \langle a, b \rangle = (|a+b|^2)^{p/2} - |a|^p - p t |a|^{p} \\
& = & |a|^p \left( (|1+t|^2 + \frac{|c|^2}{|a|^2} \right)^{p/2} - |a|^p - p t |a|^{p} = |a|^p \psi(t,\tau^2, p),
\end{eqnarray*}
and
\begin{eqnarray*}
\Gamma(a,b,K,\eps,p) & = & |a|^p 2 Kp \left( t^2+\tau^2 \right)^{1/2} \1_{(|t|^2+\tau^2)^{1/2} \geq \vartheta(\eps,p)} + |a|^p  p\eps (t^2+\tau^2)  \1_{(|t|^2+\tau^2)^{1/2}  < \vartheta(\eps,p)} \\
 & = & |a|^p \gamma(t,\tau^2,K,\eps,p).
 \end{eqnarray*}
Hence the conclusion of the lemma holds if and only if there exists $\eps \in (0,\frac{p-1}{2})$ such that for all $t \in \R$ and $\tau^2 \geq 0$ it holds that $\psi(t,\tau^2,p)\geq  \gamma(t,\tau^2,K,\eps,p)$.

Let $h\colon (0,\infty) \to \R$ be the function satisfying for all $x\in (0,\infty)$ that
$$h(x) =\frac{1}{2^{p/2}}x^{p} -2^{p/2}- 1 - p(2K+1)x. $$
Since $p> 1$, the function $h$ tends to $+\infty$ when $x \to+ \infty$. Hence there exists a constant $\alpha(K,p) \geq 2$ such that for all $x \geq \alpha(K,p)$ it holds that $h(x) \geq 0$. Now, for the sequel of the proof, fix $\eps \in (0,\frac{p-1}{2})$ such that $\vartheta(\eps,K)\ge \alpha(K,p)$.

First case: Let $t\in \R$ and $\tau^2 \ge 0$ such that $(t^2+\tau^2)^{1/2}<\vartheta(\eps,p)$.  In particular it holds that 
\begin{equation}\label{eq:up_bound_tau_1}
\tau^2<\vartheta(\eps,p)^2=\left(\sqrt{\frac{1}{2}\left(\frac{p-1}{2\eps} \right)^{\frac{2}{2-p}} + \frac{1}{2}} -1\right)^2 <  \frac{1}{2} \left[  \left(\frac{p-1}{2\eps} \right)^{\frac{2}{2-p}} - 1 \right],
\end{equation}
and consequently that
\begin{equation}\label{eq:up_bound_tau}
 \tau^2  <    \left(\frac{p-1}{2\eps} \right)^{\frac{2}{2-p}} - 1=\left( \frac{1}{p\eps} \right)^{\frac{2}{2-p}} \wedge \left(\frac{p-1}{2\eps} \right)^{\frac{2}{2-p}} \wedge \left(\frac{p-1}{2\eps} \right)^{\frac{2}{2-p}} - 1 \wedge   \left(\frac{1}{2\eps} \right)^{\frac{2}{2-p}} - 1 
 \end{equation}
Moreover, it holds that
\begin{equation}
\eps < \frac{p-1}{2}=\frac{p-1}{2} \wedge \frac{p-1}{p} \wedge \frac{1}{2}.
\end{equation}
We have to show that $\psi(t,\tau^2,p)\ge p\eps (t^2+\tau^2)$.  To this end we consider the function $\sigma \colon \R \to \R$,
$$\sigma(s) = \psi(s,\tau^2,p) -  p\eps (s^2+\tau^2) =  \left( (1+s)^2 + \tau^2 \right)^{p/2} - 1 - p s  - p\eps (s^2+\tau^2)$$
for $s\in \R$.
The first and second derivatives of $\sigma$ are given by
\begin{eqnarray*}
\sigma^\prime(s) & = & p \left( (1+s)^2 + \tau^2 \right)^{p/2-1} (1+s) - p  - 2p\eps s,\\
\sigma''(s) & = & p \left( (1+s)^2 + \tau^2 \right)^{p/2-2} \left( (p-1)(1+s)^2+\tau^2 \right) - 2p\eps. 
\end{eqnarray*}
Observe that $\sigma^\prime(-1) = -p+2p\eps$ and 
$$\sigma(-1) = \tau^p - 1 + p - p\eps (1+\tau^2) = p-1-p\eps + \tau^p(1-p\eps \tau^{2-p}).$$
Since $\eps < \frac{p-1}{p} \wedge \frac{1}{2}$ and $\tau^2 \leq \left( \frac{1}{p\eps} \right)^{\frac{2}{2-p}}$, it holds that $\sigma(-1) > 0$ and $\sigma^\prime(-1) < 0$. 

Note that $\sigma''$ is well defined for all $s \neq -1$. Moreover since $p-1<1$, for any $s\neq -1$ it holds that
$$\sigma''(s) \geq  p(p-1) \left( (1+s)^2 + \tau^2 \right)^{p/2-1} - 2p\eps =: g(s).$$
Since $\tau^2 < \left(\frac{p-1}{2\eps} \right)^{\frac{2}{2-p}}$ we have $g(s)=0$ if and only if $(1+s)^2 = \left(\frac{p-1}{2\eps} \right)^{\frac{2}{2-p}} - \tau^2$.

Observe that
$$ \Xi(\tau^2,\eps,p)= -1 - \sqrt{\left(\frac{p-1}{2\eps} \right)^{\frac{2}{2-p}} - \tau^2} <-1.$$
is the only root of $g$ on $(-\infty,-1)$. Moreover, it holds that $g(-1) > 0$. Hence at least on the interval $(\Xi(\tau^2,\eps,p),-1)$, $\sigma'' > 0$. Thus $\sigma^\prime$ is increasing on $(\Xi(\tau^2,\eps,p),-1)$ with $\sigma^\prime(-1) < 0$; in other words $\sigma$ is decreasing on $(\Xi(\tau^2,\eps,p),-1)$ and $\sigma(-1) > 0$. Thereby $\sigma(s) > 0$ for all $s\in (\Xi(\tau^2,\eps,p),-1]$. 
Observe that $(t^2+\tau^2)^{1/2}<\vartheta(\epsilon,p)$ implies that 
$$t>-\vartheta(\epsilon,p)=-\sqrt{\frac{1}{2}\left(\frac{p-1}{2\eps} \right)^{\frac{2}{2-p}} + \frac{1}{2}} +1>-\sqrt{\frac{1}{2}\left(\frac{p-1}{2\eps} \right)^{\frac{2}{2-p}} + \frac{1}{2}} -1  >\Xi(\tau^2,\eps,p),$$
where we used \eqref{eq:up_bound_tau_1} for the last inequality. Thus we obtain that $\sigma(t)>0$ if $t\le -1$.

Next assume that $t > -1$. Observe that 
$$\Upsilon(\tau^2,\eps,p) =  \sqrt{\left(\frac{p-1}{2\eps} \right)^{\frac{2}{2-p}} - \tau^2} -1  > -1$$
is the only root of $g$ on $(-1,\infty)$. Since $\tau^2 < \left(\frac{p-1}{2\eps} \right)^{\frac{2}{2-p}} - 1$ (see \eqref{eq:up_bound_tau}), it holds that $\Upsilon(\tau^2,\eps,p) > 0$. And on the interval $(-1, \Upsilon(\tau^2,\eps,p))$, $\sigma'' > 0$, thus $\sigma'$ is increasing there with $\sigma'(-1) < 0$. Moreover, it holds that
$$\sigma'(0) = p \left( (1+\tau^2)^{p/2-1} - 1\right) \geq 0.$$ 
Hence, there exists a value $\delta = \delta(\tau^2,p,\eps)$ in $(-1,0]$ such that $\sigma'(\delta(\tau^2,p,\eps))=0$. And on the interval $(-1, \Upsilon(\tau^2,\eps,p))$, the function $\sigma$ has a unique minimum $m$ given by $m=\sigma(\delta(\tau^2,p,\eps))$. We want to prove that $m\geq 0$. By the very definition 
$$\sigma^\prime(\delta) = p \left( (1+\delta)^2 + \tau^2 \right)^{p/2-1} (1+\delta) - p  - 2p\eps \delta = 0$$
hence
$$\left( (1+\delta)^2 + \tau^2 \right)^{p/2} = \frac{1 + 2\eps \delta }{1+\delta} \left( (1+\delta)^2 + \tau^2 \right).$$
This gives that
\begin{eqnarray*}
m & = & \sigma(\delta) = \left( (1+\delta)^2 + \tau^2 \right)^{p/2} - 1 - p \delta  - p\eps (\delta^2+\tau^2) \\
& = & \frac{1 + 2\eps \delta }{1+\delta} \left( (1+\delta)^2 + \tau^2 \right)- 1 - p \delta  - p\eps (\delta^2+\tau^2) \\
& = & (2-p)\eps \delta^2 + (2\eps + 1-p)\delta + (2-p)\eps \tau^2 + \tau^2 \frac{1-2\eps}{1+\delta}.
\end{eqnarray*}
If $\varpi$ is the function defined on $(-1,0]$ by
$$\varpi(x) = (2-p)\eps x^2 + (2\eps + 1-p)x + (2-p)\eps \tau^2 + \tau^2 \frac{1-2\eps}{1+x},$$
this function $x \mapsto \varpi(x)$ has a positive second derivative and since $\eps < \frac{p-1}{2}$, the first derivative is negative on $(-1,0]$. Hence this is a decreasing function and we obtain that for any $x \in (-1,0]$, $\varpi(x) \geq \varpi(0)$. Now $\varpi(0) = \tau^2 (1-p\eps) \geq 0$. Thus $m = \sigma (\delta) = \varpi(\delta) \geq 0$. Consequently, $\sigma$ is nonnegative on $(-1, \Upsilon(\tau^2,\eps,p))$.
Finally, observe that 
$$ t<\vartheta(\epsilon,p)=\sqrt{\frac{1}{2}\left(\frac{p-1}{2\eps} \right)^{\frac{2}{2-p}} + \frac{1}{2}} -1< \Upsilon(\tau^2,\eps,p),$$
where we used \eqref{eq:up_bound_tau_1} for the last inequality. This implies that $\sigma(t)\ge 0$ also in the case $t>-1$.  

Second case: Let $t\in \R$ and $\tau^2 \ge 0$ such that $(t^2+\tau^2)^{1/2}\ge \vartheta(\epsilon,p)$.  
We have to show that $\psi(t,\tau^2,p)\ge 2 K p(t^2+\tau^2)^{1/2}$. First observe that
\begin{eqnarray} \nonumber
\psi(t,\tau^2,p) -  2Kp\left( t^2 + \tau^2 \right)^{1/2} & = & \left( (1+t)^2 + \tau^2 \right)^{p/2} - 1 - p t  - 2Kp \left( t ^2 + \tau^2 \right)^{1/2} \\ \label{eq:tech_ineq_1008_bis}
& \geq & \left( (1+t)^2 + \tau^2 \right)^{p/2} - 1 - p(2K+1)\left( t ^2 + \tau^2 \right)^{1/2}. 
\end{eqnarray}
Now for any $t\in \R$, $(1+t)^2 \geq (t^2/2)-2$, thus for $t^2+\tau^2 \geq 4$
\begin{equation}\label{eq:tech_ineq_1008}
\left( (1+t)^2 + \tau^2 \right)^{p/2}  \geq \left( \frac{t^2}{2}  + \tau^2 -2 \right)^{p/2} = \left( \frac{t^2+\tau^2}{2}  + \frac{\tau^2}{2} -2 \right)^{p/2}.
\end{equation}
We define the function $\varrho$ on $[2,+\infty)$ by
$$\varrho(x) =  \left( x  + \frac{\tau^2}{2} -2 \right)^{p/2} - x^{p/2}.$$
This function tends to zero when $x$ goes to infinity. If $\tau^2 \geq 4$, then immediately $\varrho(x) \geq 0$. If not, $\varrho$ is non decreasing and $\varrho(x) \geq \varrho(2) \geq -2^{p/2}$. In any case, for $x \geq 2$, $\varrho(x) + 2^{p/2} \geq 0$. With $x = (t^2+\tau^2)/2 \geq 2$, from \eqref{eq:tech_ineq_1008} we obtain that 
$$\left( (1+t)^2 + \tau^2 \right)^{p/2}  \geq \left( \frac{t^2+\tau^2}{2}  \right)^{p/2} - 2^{p/2}$$
and therefore from \eqref{eq:tech_ineq_1008_bis} if $t^2+\tau^2 \geq 4$
\begin{eqnarray*}
\psi(t,\tau^2,p) -  2Kp\left( t^2 + \tau^2 \right)^{1/2} 
& \geq & \frac{1}{2^{p/2}}\left( t^2 + \tau^2 \right)^{p/2} -2^{p/2}- 1 - p(2K+1)\left( t^2 + \tau^2 \right)^{1/2}\\
& =& h((t^2+\tau^2)^{1/2}).
\end{eqnarray*}
Since, we chose $\eps \in (0,\frac{p-1}{2})$ such that $\vartheta (\eps,K)\ge \alpha(K,p)$ and it holds that $h\ge 0$ on $(\alpha(K,p),\infty)$, it follows that $h((t^2+\tau^2)^{1/2})\ge 0$ and hence $\psi(t,\tau^2,p)\ge 2p K (t^2+\tau^2)^{1/2}$. This completes the proof.
\end{proof}

Even if we can not compute $\alpha(K,p)$ explicitely, one can take 
$$\alpha(K,p) =  \left( 4(2K+2) + 1\right)^{\frac{1}{p-1}}.$$
And thus $\vartheta(\eps,p) \geq \alpha(K,p)$ if
\begin{equation} \label{eq:estim_eps}
\eps \leq \frac{p-1}{2 \left( \alpha(K,p) + 1 \right)^{2-p}} .
\end{equation}
The right-hand side is a decreasing function w.r.t. $p \in (1,2)$ and w.r.t. $K\geq 0$. Hence when $p$ is close to one and $K$ is large, $\eps$ is be very small.

\subsection*{Acknowledgements.} 
A. Popier thanks Laurent Denis and Sa\"id Hamad\`ene for the fruitful discussions on this topic. The authors are grateful to Bruno Bouchard, Dylan Possama\"i, Xiaolu Tan and Chao Zhou for indicating us our mistake in the first version. The authors thank sincerely the anonymous referees for helpful comments and suggestions.

\bibliography{biblio_revised_version}

\end{document}